\documentclass[12pt]{article}
\usepackage[margin=0.8in]{geometry}
\usepackage{algcompatible}
\usepackage{caption}
\providecommand{\keywords}[1]{\small \quad \quad \textbf{Keywords:} #1}
\usepackage{titlesec}
\usepackage[hidelinks]{hyperref}
\hypersetup{colorlinks=true, linkcolor=blue, filecolor=magenta, urlcolor=cyan, citecolor=blue}
\usepackage[affil-it]{authblk}
\usepackage{graphicx}
\usepackage{multirow}
\usepackage{amsmath,amssymb,amsfonts}
\usepackage{amsthm}
\usepackage{mathrsfs}
\usepackage[title]{appendix}
\usepackage{xcolor}
\usepackage{textcomp}
\usepackage{manyfoot}
\usepackage{booktabs}
\usepackage{longtable}
\usepackage{nicematrix}
\usepackage{natbib}
\usepackage[linesnumbered,ruled,vlined]{algorithm2e}

\SetCommentSty{AlgoCommentStyle}
\usepackage{algorithmicx}
\usepackage{algpseudocode}
\usepackage{listings}
\usepackage{xspace}
\usepackage{dsfont}
\usepackage{cleveref}
\usepackage{amsmath}
\usepackage{tikz}
\tikzstyle{vertex}=[draw,circle,minimum size=18pt,inner sep=0pt]
\usepackage[font=small, labelfont=bf, margin=1cm]{caption}
\usepackage{subcaption}
\usepackage{tabularx}
\usepackage{url}
\usepackage{makecell}
\usepackage{orcidlink}

 \newcommand{\tbS}{\texttt{S}\xspace}
\newcommand{\tbT}{\texttt{T}\xspace}
\newcommand{\tbG}{\texttt{G\,\%}\xspace}
\newcommand{\tbST}{\texttt{CT}\xspace}

\newcommand{\tbAUTO}{\texttt{AUTO-BD}\xspace}
\newcommand{\tbCPX}{\texttt{CPX}\xspace}
\newcommand{\tbNDO}{\texttt{ND}\xspace}
\newcommand{\tbDO}{\texttt{$\rm \Delta$O(\%)}\xspace}
\newcommand{\tbPEPF}{\texttt{\BnC}\xspace}
\newcommand{\tbCON}{\texttt{C}\xspace}
\newcommand{\tbET}{\texttt{ET}\xspace}
\newcommand{\tbBnBC}{\texttt{\BD}\xspace}
\newcommand{\tbN}{\texttt{N}\xspace}

\newcommand{\tbGC}{\texttt{GI\,\%}\xspace}

\newcommand{\tbBnBCC}{\texttt{NoCI}\xspace}
\newcommand{\tbBnBCCL}{\texttt{BD+LCI}\xspace}

\newcommand{\tbNCI}{\texttt{CI}\xspace}
\newcommand{\tbNLCI}{\texttt{LCI}\xspace}

\newcommand{\CPU}{\text{CPU}\xspace}
\newcommand{\BnBC}{{B\&BC}\xspace}
\newcommand{\BnC}{{B\&C}\xspace}
\newcommand{\LP}{{LP}\xspace}

\newcommand{\MIP}{{MIP}\xspace}
\newcommand{\MINLP}{{MINLP}\xspace}

\newcommand{\SCP}{{SCP}\xspace}

\newcommand{\CCSCP}{{PSCP}\xspace}
\newcommand{\CCSCPs}{{PSCPs}\xspace}

\newcommand{\BD}{{BD}\xspace}

\newcommand{\KKT}{{KKT}\xspace}
\newcommand{\SAA}{{SAA}\xspace}
\newcommand{\T}{\top}
\newcommand{\ie}{i.e.,\ }

\newcommand{\stt}{\,:\,}

\DeclareMathOperator{\conv}{conv}
\DeclareMathOperator{\proj}{Proj}

\newcommand{\M}{\mathcal{M}}
\newcommand{\X}{\mathcal{X}}
\newcommand{\Y}{\mathcal{Y}}

\newcommand{\setN}{\mathcal{S}}
\newcommand{\setS}{\mathcal{S}}

\newcommand{\setF}{\mathcal{F}}

\newcommand{\setC}{\mathcal{C}}
\newcommand{\CQ}{\mathcal{Q}}
\newcommand{\setL}{\mathcal{L}}
\newcommand{\SN}{\mathcal{N}}

\newcommand{\Prb}{\mathbb{P}}
\newcommand{\red}[1]{{\color{black}{#1}}}

\newcommand{\rev}[1]{{\color{black}{#1}}}

\newcommand{\cplex}{{CPLEX}\xspace}

\theoremstyle{plain}
\newtheorem{Theorem}{Theorem}[section]
\newtheorem{corollary}[Theorem]{Corollary}
\newtheorem{lemma}[Theorem]{Lemma}
\newtheorem{remark}[Theorem]{Remark}

\newtheorem{definition}[Theorem]{Definition}

\newtheorem{proposition}[Theorem]{Proposition}

 \crefname{theorem}{Theorem}{Theorems}
\crefname{example}{Example}{Examples}
\crefname{observation}{Observation}{Observations}
\crefname{remark}{Remark}{Remarks}
\crefname{proposition}{Proposition}{Propositions}
\crefname{property}{Property}{Propertys}
\crefname{lemma}{Lemma}{Lemmas}
\crefname{corollary}{Corollary}{Corollaries}
\crefname{algocf}{Algorithm}{Algorithms}	
\crefname{table}{Table}{Tables}	
\crefname{figure}{Figure}{Figures}
\crefname{algorithm}{Algorithm}{Algorithms}
\crefname{section}{Section}{Sections}

\title{Towards large-scale probabilistic set covering problems: an efficient Benders decomposition approach}

\author[a]{Wei Lv\orcidlink{0009-0009-6861-9532}}
\author[a]{Wei-Kun Chen\orcidlink{0000-0003-4147-1346}}
\author[a]{Yi-Long Chen}
\author[b,c]{Yu-Hong Dai\orcidlink{0000-0002-6932-9512}}

\affil[a]{\small School of Mathematics and Statistics, Beijing Institute of Technology, Beijing 100081, China\\\textit{\{lvwei,chenweikun,chenyilong\}@bit.edu.cn}}
\affil[b]{\small State Key Laboratory of Mathematical Sciences, Academy of Mathematics and Systems Science, Chinese Academy of Sciences, Beijing 100190, China}
\affil[c]{\small School of Mathematical Sciences, University of Chinese Academy of Sciences, Beijing 100049, China\\\textit{dyh@lsec.cc.ac.cn}}
\date{\small \today}

\begin{document}
	
\maketitle

\begin{abstract}
In this paper, we investigate the probabilistic set covering problem (\CCSCP) in which the right-hand side is a binary random vector and the covering constraint is required to be satisfied with a prespecified probability.
We consider the case with a finite discrete distribution of the random vector,  which usually arises in the context of the {sample average approximation approach}. 
We develop an effective Benders decomposition (\BD) algorithm for solving large-scale \CCSCPs, which enjoys two key advantages: (i) the number of variables in the underlying Benders reformulation is independent of the scenario size; 
and (ii) the Benders cuts can be separated by an efficient combinatorial algorithm.
For the special case that the random vector is a combination of several independent random blocks/subvectors, we explicitly take this kind of block structure into consideration and develop a more efficient \BD algorithm.  
Moreover, to further speed up the two proposed \BD algorithms, we develop a class of strong valid inequalities, which are guaranteed to be facet-defining for the polytope induced by the probabilistic constraint.
Numerical results on instances with up to one million scenarios demonstrate the effectiveness of the proposed \BD algorithms over a black-box mixed integer programming solver's branch-and-cut and automatic \BD algorithms and a state-of-the-art algorithm in the literature.
	\vspace{8pt} \\
\end{abstract}

\section{Introduction}\label{sec:intro}

In this paper, we consider the probabilistic set covering problem (\CCSCP) of the form:
\vspace*{0.2cm}
\begin{align}
	\label{ccscp}\tag{PSCP}
		\min \left\{c^\T x \,:\, \Prb \{Ax \geq \xi\} \geq 1 - \epsilon,~x \in \{0,1\}^{n} \right\},
\end{align}
where $A$ is a 0-1 matrix of dimension $m \times n$, 
$c \in \mathbb{R}^{n}_+$ is the cost of columns, 
$\epsilon$  is a confidence parameter chosen by the decision maker, typically near zero, 
and $\xi$ is a random vector taking values in $\{0,1\}^m$.
\eqref{ccscp} generalizes the well-known set covering problem (\SCP) \citep{Toregas1971} where the random vector $\xi$ is replaced by the all ones vector and has obvious applications in location problems \citep{Beraldi2002,Norkin2018}. 
It also serves as a building block in various applications; see  \cite{Saxena2010,Wu2016,Azizi2022,Feng2024}.

\eqref{ccscp} falls in the class of probabilistic programs (also known as chance-constrained programs), which has been extensively investigated in the {literature}.
We refer to the surveys \cite{Ahmed2008,Kucukyavuz2017,Kucukyavuz2022,Prekopa2003} for a detailed discussion of probabilistic programs and the various algorithms.
\cite{Beraldi2002} first investigated \eqref{ccscp} and developed a specialized branch-and-bound algorithm, which involves enumerating the so-called $(1-\epsilon)$-efficient points introduced by \cite{Prekopa1990} and solving a deterministic \SCP for each $(1-\epsilon)$-efficient point.
\cite{Saxena2010} studied the case with the assumption that the random vector $\xi$ can be decomposed into $T$ blocks/subvectors $\xi^1,\ldots,\xi^T$ formed by subsets $\M^t \subseteq [m]$, and $\xi^{t_1}$ and $\xi^{t_2}$ are independent for any two distinct $t_1, t_2 \in [T]$ (throughout the paper, for a nonnegative integer $\tau$, we denote $[\tau] = \{1,\ldots,\tau\}$ where $[\tau] = \varnothing$ if $\tau = 0$).
{Note that this assumption is not restrictive as it encompasses the cases of all independent and dependent random variables \citep{Beraldi2002}.}
Problem  \eqref{ccscp} in this special case can be rewritten as
\begin{align}
	\label{biccscp}\tag{PSCP'}
	\begin{aligned}
		\min \left\{c^\T x \, : \, \prod_{t=1}^T \Prb \{A^t x \geq \xi^t\} \geq 1 - \epsilon, ~ x \in \{0,1\}^{n} \right\},
	\end{aligned}
\end{align} 
where $A^t$ is the submatrix of $A$ formed by the rows in $\M^t$, $t \in [T]$.
\cite{Saxena2010} developed an equivalent mixed integer programming (\MIP) formulation for \eqref{biccscp} based on the $(1-\epsilon)$-efficient and -inefficient points of the distribution of $\xi^t$ and solved it by a customized branch-and-cut (\BnC) algorithm.
For set covering models with uncertainty on the constraint matrix $A$, we refer to \cite{Ahmed2013,Beraldi2010,Fischetti2012,Hwang2004,Jiang2025,Lutter2017,Shen2023,Song2013,Wu2019} and the references therein.

One weakness of the above two approaches for solving \eqref{ccscp} is that \rev{they require} an enumeration of the $(1-\epsilon)$-efficient (and -inefficient) points of the distribution of $\xi$ or $\xi^t$, which could be computationally demanding, especially when the dimensions are large.
In addition, the \MIP formulation of  \cite{Saxena2010} may grow exponentially with the dimensions of $\xi^t$, making it only capable of solving \eqref{ccscp} with small dimensions of $\xi^t$.
\cite{Luedtke2008} addressed this difficulty by using the sample average approximation (\SAA) approach \citep{Luedtke2008,Pagnoncelli2009}, where \eqref{ccscp} is replaced by an approximation problem based on Monte Carlo samples of the random vector $\xi$.
The approximation problem is also a form of \eqref{ccscp} but has a finite discrete distribution of $\xi$.
As such, it can be reformulated as a deterministic \MIP problem and solved by the state-of-the-art \MIP solvers.
In general, the selection of the scenario size is crucial for the approximation quality;
the larger the scenario size, the smaller the approximation error; see \cite{Ahmed2008,Luedtke2008,Nemirovski2006a} for related discussions. 
However, since the problem size of the \MIP formulation grows linearly with the number of scenarios, this also poses a new challenge for \MIP solvers to solve \eqref{ccscp} with a huge number of scenarios.

\subsection{Contributions and outlines}

The main motivation of this paper is to investigate decomposition based algorithms to solve large-scale \CCSCPs.
In particular, 
\begin{itemize}
	\item [(i)] For \CCSCPs without the consideration of a block structure of $\xi$, i.e., problem \eqref{ccscp}, we investigate the \MIP formulation of \cite{Luedtke2008} and propose a customized  single-cut Benders decomposition (\BD) algorithm where the scenario variables are projected out from the problem formulation and the Benders cuts are separated on the fly.
	Two key advantages of the proposed \BD algorithm, which make it capable of solving \eqref{ccscp} with a huge number of scenarios, are: (i) the number of variables in the underlying Benders reformulation is equal to $m+n$ but is independent of the number of scenarios of the random data; and (ii) the Benders cuts can be separated by an {efficient combinatorial algorithm.}
	\item [(ii)] For \CCSCPs with a block structure of $\xi$ \citep{Saxena2010}, i.e., problem \eqref{biccscp}, we develop a mixed integer nonlinear programming (\MINLP) formulation, and propose a multi-cut customized \BD algorithm for it. 
	The proposed multi-cut \BD algorithm explicitly takes the block structure of the random vector $\xi$ into account, and adds up to $T$ Benders cuts in a single iteration, which renders the \BD algorithm able to converge more quickly.
	With this feature, we show that although both \MIP and \MINLP formulations can be seen as approximations of the \CCSCP in the context of the \SAA approach, 
	the multi-cut \BD algorithm for the \MINLP formulation is even more efficient than the single-cut \BD algorithm for the \MIP formulation. 
	\item [(iii)] To further speed up the two proposed \BD algorithms, we leverage the sequential lifting technique \citep{Wolsey1976,Richard2011} and develop a class of strong valid inequalities---lifted cover inequalities, which, unlike the Benders cuts, are guaranteed to be facet-defining for the polytope induced by the probabilistic constraint.
	Moreover, to alleviate the potential computational burden spent in the lifting procedure, we investigate a subclass of them---the maximal clique inequalities, which enjoy light-weight algorithms for their computation and separation.
	Equipped with the maximal clique inequalities, the two proposed \BD algorithms converge much faster, especially for instances with a small confidence parameter $\epsilon$.
\end{itemize}

By extensive computational experiments on a testset of \rev{$7200$} \CCSCP instances with up to $10^6$ scenarios, we demonstrate
the effectiveness of our proposed \BD algorithms for solving large-scale \CCSCPs over existing state-of-the-art approaches.
In particular, for \CCSCPs without the consideration of a block structure of the random vector $\xi$, 
the proposed \BD algorithm is much more efficient than a state-of-the-art \MIP solver's \BnC and automatic \BD algorithms \rev{(when the number of scenarios is large)};
for \CCSCPs with a block structure of $\xi$, the proposed \BD algorithm is much more robust than the state-of-the-art approach in \cite{Saxena2010} in terms of solving \eqref{biccscp} with different input parameters.

It deserves to mention that although the \BD algorithms are proposed to tackle \CCSCPs, 
they can also be applied to solve other variants of \CCSCPs that involve the probabilistic covering constraint $ \Prb \{Ax \geq \xi\} \geq 1 - \epsilon$ where $\xi$ is a binary random vector.
In Online Appendix A, we consider three variants of the \CCSCP investigated in \cite{Saxena2010} (i.e., the probabilistic versions of the single-source capacitated facility location problem, capacitated warehouse location problem, and capacitated $p$-median problem), and demonstrate the efficiency of the proposed \BD algorithms over existing state-of-the-art approaches.

\rev{\cite{Luedtke2014} and \cite{Liu2016} developed decomposition algorithms along with the mixing inequalities \citep{Atamturk2000,Gunluk2001b,Luedtke2010} for a more general setting---the two-stage chance-constrained programs---in which the left-hand side of the constraints may also involve uncertainty. 
	The decomposition algorithms and mixing inequalities, however, are fundamentally different from the ones proposed in this paper. 
	First, the decomposition algorithms of \cite{Luedtke2014} and \cite{Liu2016} attempt to project the second stage variables out from the problem formulation but still involve scenario variables.
	In contrast, the proposed \BD algorithms exploit the structural properties of the \CCSCP and project the scenario variables out from the problem formulations, 
	making them capable of solving \CCSCPs with a huge number of scenarios.
	Second, {the mixing inequalities} of \cite{Luedtke2014} and \cite{Liu2016} are generated from the so-called mixing set and involve the scenario variables.
	However, the proposed lifted cover inequalities are derived from the polytope induced by the probabilistic constraint and do not involve the scenario variables, rendering  them able to be embedded in the proposed \BD algorithms.}
	
An extended abstract of this paper can be found in 2024 INFORMS Optimization Society (IOS) conference \citep{Chen2024}.
Compared to its conference version, this paper investigates the polyhedral structures of the $0$-$1$ sets arising in \CCSCPs with and without a block structure of $\xi$, and develops strong valid inequalities to improve the computational performance of the \BD algorithms for solving \CCSCP{s}.
Moreover, extensive computational results are also reported in this paper.

The rest of the paper is organized as follows. 
\cref{sec:formulation} presents the mathematical formulations for \eqref{ccscp} and \eqref{biccscp}, and provides a favorable property for the formulations. 
\cref{sec:bd} develops the \BD algorithms for solving the two formulations. 
\cref{sec:lci} derives the strong valid inequalities to improve the performance of the \BD algorithms.
\cref{section:num} presents the computational results.
Finally, \cref{section:conclusion} draws the conclusion.

\section{Mathematical formulations}\label{sec:formulation}

Consider \eqref{ccscp} with an (arbitrary) finite discrete distribution of the random vector $\xi$, that is, 
there exist finitely many scenarios $\xi_1, \ldots, \xi_s \in \{0,1\}^m$ such that $\mathbb{P}\{\xi = \xi_i\} = p_i$ for $i \in[s]$ and $\sum_{i = 1}^{s} p_i = 1$.
We introduce for each $k\in [m]$, a binary variable $v_k$, where $v_k=1$ guarantees $A_k x \geq 1$,  and for each $i \in [s]$, a binary variable $z_i$, where $z_i = 1$ guarantees $ v \geq \xi_i$.
Then \eqref{ccscp} can be formulated as the following \MIP formulation \citep{Luedtke2008}: 
\begin{align}
	\min \left\{c^\T x\, : \, \eqref{cons:xandv}\text{--}\eqref{cons:vandzbound}\right\}\label{mip}\tag{MIP},
\end{align}
where
	\begin{align}
		& A x \geq v, \label{cons:xandv}\\
		 & v_k \geq z_i,~\forall~k \in [m],~i \in \setN_k, \label{cons:vandz}\\
		 & \sum_{i = 1}^{s}p_i z_i \geq 1-\epsilon, \label{cons:knap}\\
		 & x \in \{0, 1\}^n,~ v \in \{0, 1\}^m,~z \in \{0, 1\}^s \label{cons:vandzbound}.
	\end{align}
Here, $\setN_k = \left\{i \in [s] \stt \xi_{ik} = 1\right\}$, $k \in [m]$.
Constraints \eqref{cons:vandz}--\eqref{cons:vandzbound} ensure $\mathbb{P} \{v \geq \xi\} \geq 1 - \epsilon$.
Note that if $\xi_i= \boldsymbol{0}$ for some $i \in [s]$, we can set  $z_i = 1$ and remove variable $z_i$ from the formulation. 
Therefore, we can assume $\xi_i \neq \boldsymbol{0}$ for all $i \in [s]$ and thus $\bigcup_{k \in [m]} \setN_k = [s]$ in the following.

For problem \eqref{biccscp} (i.e., problem \eqref{ccscp} with a block structure of $\xi$) with a finite discrete distribution of $\xi^t$: $\mathbb{P}\{\xi^t = \xi^t_i\} = p_{it}$ for $i \in[s]$ and $\sum_{i = 1}^{s} p_{it} = 1$ (for simplicity of exposition, here we assume that the numbers of scenarios of $\xi^t$ are identical for all $t \in  [T]$; the extension to the case that the numbers of scenarios of $\xi^t$ are different for different $t$ is straightforward), 
a similar convex \MINLP problem can be presented as follows:
\begin{align}
	\min \left\{c^\T x\, : \, \eqref{cons:xandv},~ \eqref{cons:vandz2}\text{--}\eqref{cons:vandzbound2}\right\}\label{minlp}\tag{MINLP},
\end{align}
where
	\begin{align}
		 & v_k \geq z_{it},~\forall~t \in [T],~k \in \M^t,~i \in \setN^t_k,\label{cons:vandz2}\\
		 & \eta_t \leq \ln\left(\sum_{i=1}^s p_{it} z_{it}\right),~\forall~t \in [T], ~\sum_{t=1}^T \eta_t \geq \ln(1-\epsilon),\label{cons:eta}\\
		 & x \in \{0, 1\}^n,~v \in \{0, 1\}^m,~z \in \{0,1\}^{s\times T},~\eta\in \mathbb{R}^T.\label{cons:vandzbound2}
	\end{align}
Here,  $\setN^t_k = \left\{i \in [s] \stt \xi^t_{ik} = 1\right\}$ (similarly, we assume $\bigcup_{k \in \M^t} \setN^t_k = [s]$); 
variable $\eta_t$ represents the value of $\ln(\Prb\{v^t \geq \xi^t\})$, where $v^t$ is a subvector of $v$ formed by the rows in $\M^t$; and constraints \eqref{cons:vandz2}--\eqref{cons:vandzbound2} ensure that $ \prod_{t=1}^T \Prb \{v^t \geq \xi^t\} \geq 1 - \epsilon$.
\begin{remark}\label{independentRem}
	By $ \prod_{t=1}^T \Prb \{v^t \geq \xi^t\} \geq 1 - \epsilon$, $\Prb \{v^t \geq \xi^t\} \geq 1 - \epsilon$ must hold for all $t \in [T]$.
	As a result, for any feasible solution {$(x,v,z,\eta)$} of problem \eqref{minlp}, it follows
	\begin{equation}\label{independentEq}
		\sum_{i=1}^s p_{it} z_{it} \geq 1-\epsilon, ~\forall~t \in [T]. 
	\end{equation}
\end{remark}

Note that if the random vector $\xi$ does not have a finite discrete distribution, both problems \eqref{mip} and \eqref{minlp} can be seen as approximations of problem \eqref{ccscp} in the context of the \SAA approach, where $\{\xi_i\}_{i \in [s]}$ and $\{\xi_i^t\}_{(i,t) \in[s]\times[T]}$ are independent Monte Carlo samples of the random vector $\xi$.  
\rev{When the number of blocks $T$} is equal to one, problems \eqref{mip} and \eqref{minlp} with the same samples (i.e., $\xi_i= \xi_i^1$ for all $i \in [s]$) are equivalent.
However, \rev{when the number of blocks $T$} is larger than one, the two problems with the same random samples (i.e., $\xi_i= \left[{(\xi_i^1)}^\top, \ldots, {(\xi_i^T)}^\top\right]^\top$ for all $i \in [s]$) are not equivalent in general as the random vector  $\xi$ with a finite discrete distribution, obtained from Monte Carlo sampling, may not have an independent block structure; see Online Appendix B for an illustrative example.
In \cref{sec:comparetwoproblem}, we will perform computational experiments to compare the performance of the two approximations.

Although problems \eqref{mip} and \eqref{minlp} can be solved to optimality by the state-of-the-art \MIP/\MINLP solvers, the potentially huge numbers of scenario variables $z$ and related constraints make this approach only able to solve moderate-sized instances.
Indeed, for problem  \eqref{mip}, the numbers of variables and constraints are $m+n+s$ and $m+1+\sum_{k =1}^m |\setS_k|$, respectively; for problem \eqref{minlp}, the numbers of variables and constraints are $m+n+(s+1)T$ and $m+T+1+\sum_{t=1}^T \sum_{k \in \M^t} |\setS_k^t|$, respectively.
To alleviate the computational difficulty, \cite{Lejeune2008,Lejeune2010,Luedtke2008} proposed a preprocessing technique to reduce the numbers of variables and constraints in problem \eqref{mip}.
In particular, for problem \eqref{mip}, if $p(\setN_k):= \sum_{i \in \setN_k} p_i > \epsilon$, then it follows from \eqref{cons:vandz} and \eqref{cons:knap} that $v_k=1$ must hold and the related constraints $v_k \geq z_i$ for $i \in \setS_k$ can be removed from \eqref{cons:vandz}.
The preprocessing technique can also be extended to problem \eqref{minlp}.
Unfortunately, the numbers of variables and constraints after preprocessing can still grow linearly with the number of scenarios, making it  challenging to solve problems with a huge number of scenarios.
For problem \eqref{minlp}, \cite{Saxena2010} proposed an equivalent \MIP reformulation in the space of $x$ and $v$. 
However, the problem size of this \MIP reformulation depends on the number of the so-called $(1-\epsilon)$-efficient (and -inefficient) points, which generally grows exponentially with the block sizes of $\{\xi^t\}$. 

To overcome the above weakness, we will develop an efficient \BD approach in \cref{sec:bd},
which is based on a favorable property of the two formulations, detailed as follows.
Let (\MIP') (respectively, (\MINLP')) be the relaxation of problem \eqref{mip} (respectively, \eqref{minlp}) obtained by removing the integrality constraints $z \in \{0,1\}^s$ (respectively, $z \in \{0,1\}^{s\times T}$) and replacing the integrality constraints $v \in \{0,1\}^m$ by $v \leq \boldsymbol{1}$.
The following proposition shows that this operation does not change the \rev{optimal} values of problems \eqref{mip} and \eqref{minlp}.
\begin{proposition} \label{property:relax}
	Problems \eqref{mip} and \eqref{minlp} are equivalent to {\rm (\MIP'\rm)} and {\rm (\MINLP')}, respectively, in terms of sharing the same optimal values.
\end{proposition}
\begin{proof}{Proof}
	Let $o_{\rm \MIP}$ and $o_{\rm \MIP\text{'}}$ be the optimal values of problems \eqref{mip} and (\MIP').
	Clearly, $o_{\rm \MIP\text{'}} \leq o_{\rm \MIP}$.
	To prove the other direction, suppose that $(x,v,z)$ is an optimal solution of (\MIP').
	Let \rev{$\bar{v}$ and $\bar{z}$} be defined by $\bar{v} = \min\{Ax, \boldsymbol{1}\} \in \{0,1\}^m$ and $\bar{z}_{i} =  \min_{k \in \M_i} \bar{v}_k\in \{0,1\}$, where $\M_i = \{ k \in [m]\,:\, i \in \setN_k  \}$, $i \in [s]$.
	By \eqref{cons:xandv}, it follows $v \leq \min\{Ax, \boldsymbol{1}\}= \bar{v}$, and 
	by \eqref{cons:vandz}, it follows $z_i \leq \min_{k \in \M_i} v_k \leq  \min_{k \in \M_i} \bar{v}_k = \bar{z}_i$.
	Thus $(x, \bar{v}, \bar{z})$ must be a feasible solution of problem \eqref{mip}. 
	As a result, $o_{\rm \MIP} = c^\top x  \leq o_{\rm \MIP\text{'}}$.
	This shows that $o_{\rm \MIP}=o_{\rm \MIP\text{'}}$ and the equivalence of problems \eqref{mip} and (\MIP').  
	The proof for the equivalence of problems \eqref{minlp} and (\MINLP') is analogous. \qedhere
\end{proof}
 
Although the integrality constraints $v \in \{0,1\}^m$ can be relaxed into $v \leq \boldsymbol{1}$, \cite{Luedtke2008} noted that leaving $v \in \{0,1\}^m$ in problem \eqref{mip} enables \BnC solvers to favor branching on variables $v$ and effectively improves the overall performance.
Our preliminary experiments also justified this statement (for both (\MIP') and (\MINLP')).
Due to this, we decide to leave the integrality constraints $v\in \{0,1\}^m$ in the two formulations.

	\section{Benders decomposition}\label{sec:bd}

In this section, we first review the (generalized) \BD approach for generic convex \MINLP{s} and then apply the \BD approach to  formulations \eqref{mip} and \eqref{minlp} of \CCSCPs.

\subsection{Benders decomposition for convex \MINLP{s}}

Here we briefly review the (generalized) \BD algorithm for convex \MINLP{s}. 
For more details, we refer to \cite{Geoffrion1972,Bonami2012,Belotti2013}.
We  consider a convex \MINLP of the form
\begin{equation}
	\label{g_minlp}
	\min \left\{ \rev{c^\top x} \, : \, g_\ell(x,y) \leq 0,~\forall~ \ell \in [L],~Ax \geq b, ~x \in \mathbb{Z}^{p} \times \mathbb{R}^{n-p},~y \in \mathbb{R}^\tau \right\}, 
\end{equation}
\rev{where} $g_1(x,y), \ldots, g_L(x,y)$ are convex and twice differentiable functions.
\rev{The objective function in \eqref{g_minlp}} is a linear function of variables $x$\rev{, as} this will be the case in the context of solving formulations \eqref{mip} and \eqref{minlp}.
\rev{Notice that, however, convex \MINLP problems with a general convex and twice differentiable objective function $f(x,y)$ of variables $x$ and $y$ can also be tackled by the BD algorithm \citep{Geoffrion1972,Bonami2012,Belotti2013}.}

Let $\X = \{x \in \mathbb{Z}^p \times \mathbb{R}^{n-p} \,:\, Ax \geq b\}$ and $\Y(x) = \{ y \in \mathbb{R}^\tau \,:\, g_\ell (x,y) \leq 0,\,\forall\,\ell \in [L] \}$.
Problem \eqref{g_minlp} can be formulated as the \emph{master problem} in the $x$ space:
\begin{equation}
	\label{mp}
	\min \left\{ c^\top x  \, : \,  \Psi(x) \leq 0, ~x\in \X \right\}, 
\end{equation}
where
\begin{equation}\label{FP}
		\Psi(x) = \min \left\{ \sum_{\ell \in \mathcal{L}^{\bot}} w_\ell r_\ell \stt  g_\ell (x,y) \leq 0,~\forall~\ell \in  \mathcal{L},~g_\ell (x,y) \leq r_\ell,~r_\ell \geq 0,~\forall~\ell \in \mathcal{L}^\bot\right\}
\end{equation}
is a convex function measuring the infeasibility of $g_\ell(x,y)$, $\ell \in [L]$, at point $x$ (if $\Y(x) =\varnothing$). 
Here $w_\ell >0$, $\ell \in \mathcal{L}^\bot$, are the weights that can be chosen to reduce to, e.g., ${||\cdot||}_1$ norm minimization.
For simplicity, we assume that {$\mathcal{L}\subseteq[L]$ is the set of constraints} that can be satisfied at all $x\in \X$, and its complement {$\mathcal{L}^{\bot}$ is the set of constraints} that can possibly be violated at some $x \in \X$.

The master problem \eqref{mp} can be solved as an \MIP problem by a linear programming (\LP) based \BnC approach in which linear outer approximations of  the nonlinear constraint $ \Psi(x) \leq 0$, called (generalized)  \emph{Benders feasibility cuts}, are separated on the fly.
In particular, let $x^*$ be a solution of the \LP relaxation of the current master problem. 
Due to the convexity, $\Psi (x)$ can be underestimated by a supporting hyperplane $\Psi(x^*) + \mu(x^*)^\top (x-x^*)$, so we obtain the following Benders feasibility cut 
\begin{equation}\label{BDF}
	0  \geq \Psi(x^*) + \mu(x^*)^\top (x-x^*),
\end{equation}
where 
\begin{equation}
	\mu_i(x^*) = \sum_{\ell=1}^L u_\ell^* \frac{\partial g_\ell(x^*, y^*)}{\partial x_i},~\forall~i \in [n],
\end{equation}
$y^*$ and $u^*$ are optimal primal and Lagrangian dual solutions of the convex problem \eqref{FP} with $x=x^*$, derived using the Karush-Kuhn-Tucker (\KKT) conditions; see \cite{Geoffrion1972,Bonami2012,Belotti2013} for more details.

Two remarks on the above \BD algorithm are in order.
First, if problem \eqref{FP} can be decomposed into several subproblems, then we can construct a Benders feasibility cut for each subproblem.
In this case, to speed up the convergence of the algorithm, we can implement a multi-cut version of the \BD algorithm \citep{Rahmaniani2017}, 
where in each iteration, multiple Benders feasibility cuts will be (possibly) added,  each of which corresponds to a subproblem.
Second,  if $\mathcal{L}^\bot=\{\ell'\}$ is a singleton (which will be considered in the context of solving formulations \eqref{mip} and \eqref{minlp}), we can set $w_{\ell'} =1$ and solve an equivalent problem:
\begin{equation}\label{FSB}
	\Psi(x) = \min \left\{ g_{\ell'}(x,y) \, : \, ~g_\ell (x,y) \leq 0 ,~\forall~\ell \in  [L]\backslash\{\ell'\}\right\}.
\end{equation}
The Benders feasibility cut \eqref{BDF} can then be adapted as  
\begin{equation}\label{BDF2}
	0  \geq {\Psi}(x^*) + {\mu}(x^*)^\top (x-x^*),
\end{equation}
where 
\begin{equation}
	\mu_i(x^*) = \frac{\partial g_{\ell'}(x^*, y^*)}{\partial x_i}+ \sum_{\ell \in [L]\backslash \{\ell'\}} u_\ell^* \frac{\partial g_\ell(x^*, y^*)}{\partial x_i},~\forall~i \in [n],
\end{equation}
$y^*$ and $u^*$ are optimal primal and Lagrangian dual solutions of the convex problem \eqref{FSB} with $x=x^*$.

\subsection{Benders decomposition for generic \CCSCPs}
\label{subsec:BD1}

Now, we apply the \BD approach to formulation \eqref{mip} of the \CCSCP with a finite discrete distribution of $\xi$. From \cref{property:relax}, we project the $z$ variables out from the formulation and obtain the master problem
\begin{equation}\label{mp1}
	\min \left\{c^\top x\, : \, Ax \geq v,~\Psi(v) \leq 0,~(x,v) \in \{0,1\}^{n+m}\right\}.
\end{equation}
For a solution $v^* \in [0,1]^m$ (encountered in the \BnC search tree of the \BD algorithm), there always exists a feasible solution $z^*$ satisfying \eqref{cons:vandz}, and thus $\Psi(v^*)$ can be presented as follows:
\begin{equation}
	\label{LSP}
	\Psi(v^*) = \min \left\{(1-\epsilon)-\sum_{i=1}^s p_i z_i \,:\, z_i \leq v^*_k,~\forall~k \in [m],~i \in \setN_k \right\}.
\end{equation}
Note that as shown in \cref{property:relax}, it is also possible to project the $v$ variables out from the formulation. 
However, based on our computational experience, leaving the $v$ variables (whose number $m$ is generally not large) in the master problem makes the \BD algorithm converge much faster, and thus improves the overall performance.

Let $u_{ki}$ be the dual variable associated with constraint $z_i\leq v_k^*$ in problem \eqref{LSP}, and $\{u^*_{ki}\}$ be an optimal dual solution of problem \eqref{LSP}. 
Hence, 
the Benders feasibility cut \eqref{BDF2} reduces to 
\begin{equation}
	\label{linearbfc}
	0  \geq \Psi(v^*) -\sum_{k = 1}^m \sum_{i \in \setN_k} u^*_{ki} (v_k-v_k^*).
\end{equation}

Next, we derive a closed formula for the Benders feasibility cut \eqref{linearbfc} using primal and dual optimal solutions of problem \eqref{LSP}.
The Lagrangian function corresponding to problem \eqref{LSP} reads
\begin{equation}
	L(z, u) =  (1-\epsilon)-\sum_{i=1}^s p_i z_i + \sum_{k = 1}^m \sum_{i \in \setN_k} u_{ki}(z_i - v^*_k).
\end{equation}
Without loss of generality, we assume that values $\{v_k^*\}_{k \in [m]}$ are sorted nondecreasingly, i.e., $v^*_1\leq \cdots \leq v_m^*$.
Let $\M_{i}= \{ k \in [m]\, : \, i \in \setN_k \}$.
Then an optimal solution for problem \eqref{LSP} is given by
\begin{equation}\label{Primalclosedform}
	z^*_i = \min_{k \in \M_i} v^*_{k}= v^*_{k_i},~\text{where}~k_i = \min\left\{ k \, : \, i \in \setN_k \right\}, ~\forall~i \in [s].
\end{equation}
Substituting $z^*$ into the \KKT system 
\begin{equation*}
\begin{aligned}
	&\frac{\partial L(z, u)}{\partial z_{i}}= -p_i + \sum_{k \in \M_{i}} u_{ki}=0,~\forall~i \in [s],  \\
	&  u_{ki} \geq 0,~z_{i} \leq v^*_k,~ u_{ki} (z_{i} - v^*_k) = 0,~\forall~k \in [m],~i \in \setN_k,
\end{aligned}
\end{equation*}
we derive an optimal dual solution $u^*$ as follows:
\begin{equation}
	u^*_{ki} = 
	\begin{cases}
		p_i,&~\text{if} ~k = k_i; \\[5pt]
		0,    & ~\text{otherwise},
	\end{cases}~\forall~i \in [s],~k \in \mathcal{M}_i.
\end{equation} 
Plugging $\Psi(v^*)=1-\epsilon-\sum_{i=1}^s p_i v^*_{k_i}$ and $u^*$ into \eqref{linearbfc} and regrouping the terms, we obtain the Benders feasibility cut
\begin{equation}
	\label{feasibilitycut}
	\sum_{k=1}^m p(\setN_k \backslash  \setN_{[k-1]}) v_{k} \geq 1 - \epsilon.
\end{equation}
Here $\setN_{[k-1]} = \bigcup_{k' \in [k-1]} \setN_{k'}$, $\setN_{[0]} = \varnothing$, and $p(\setN_k \backslash  \setN_{[k-1]})= \sum_{i \in \setN_k \backslash  \setN_{[k-1]}} p_i$ for $k \in [m]$. 
Given a solution $v^*\in[0,1]^m$, the above procedure
 for the separation of the Benders feasibility cuts can be implemented with the complexity of $\mathcal{O}(m\log m + \sum_{k \in [m]}|\setN_k|)$, where $\mathcal{O}(m\log  m)$ is taken by sorting $\{v^*_k\}_{k \in [m]}$.

\subsection{Benders decomposition for \CCSCPs with a block structure of $\xi$}\label{subsec:BD2}

In this subsection, we apply the \BD approach to formulation \eqref{minlp} of the \CCSCP with a block structure of $\xi$. 
Similarly, we project the $z$ variables out from the formulation and obtain the master problem 
\begin{equation}\label{mp2}
	\small
		\begin{aligned}
			\min \left\{c^\top x \,:\, Ax \geq v,~\sum_{t=1}^T\eta_t \geq \ln (1-\epsilon),~\Psi_t(v^t,\eta_t) \leq 0,~\forall~t \in [T],~(x,v) \in \{0,1\}^{n+m},~{\eta \in \mathbb{R}^T}\right\},
		\end{aligned}
\end{equation}
where for a solution $(v^*, \eta^*) \in [0,1]^m \times \mathbb{R}^T$, the Benders subproblem \eqref{FP} can be decomposed into $T$ subproblems $\{\Psi_t((v^t)^*,\eta_t^*)\}_{t \in [T]}$ given by
\begin{equation}
		\label{subnonlinear}
		\Psi_t((v^t)^*,\eta_t^*) = \min \left\{\eta^*_t-\ln\left(\sum_{i=1}^s p_{it} z_{it}\right) \,:\,  z_{it} \leq v^*_k ,~\forall~k \in \M^t,~i \in \setN^t_k \right\},~\forall~t \in [T].
\end{equation}

Let $m_t = |\M^t|$ and $\M^t_{i} = \{ k \in \M^t\, : \, i \in \setN^t_k \}$. Without loss of generality, we assume $\M^t=[m_t]$ and $v_1^* \leq \cdots \leq v_{m_t}^*$.
To derive the Benders feasibility cut, let us first consider the case that
$ \min_{k \in \M^t_{i}} v^*_{k}>0$ holds for some $i \in [s]$.
In this case, $\Psi_t((v^t)^*,\eta_t^*)<+\infty$ is well-defined and an optimal solution of problem \eqref{subnonlinear} is given by 
\begin{equation}\label{Primalclosedform2}
	z^*_{it} = \min_{k \in \M^t_{i}} v^*_{k}= v^*_{k_i},~\text{where}~k_i = \min\left\{ k \, : \, i \in \setN^t_k \right\},~\forall~i \in [s].
\end{equation}
Let $u_{ki}$ be the dual variable associated with constraint $z_{it}\leq v_k^*$ in problem \eqref{subnonlinear}, and $\{u^*_{ki}\}$ is an optimal dual solution of problem \eqref{subnonlinear}.
Then the Benders feasibility cut \eqref{BDF2} can be written as
\begin{equation}
	\label{convexbfc}
	0  \geq \Psi_t((v^t)^*, \eta_t^*)  + (\eta_t - \eta_t^*)-\sum_{k \in \M^t} \sum_{i \in \setN^t_k} u^*_{ki} (v_k-v_k^*).
\end{equation}

Next, we provide a closed formula for the Benders feasibility cut \eqref{convexbfc} using primal and dual optimal solutions of problem \eqref{subnonlinear}. 
The Lagrangian function corresponding to problem \eqref{subnonlinear} reads
\begin{equation}
	L(z_{\cdot t}, u) = \eta^*_t-\ln\left(\sum_{i=1}^s p_{it} z_{it} \right) + \sum_{k \in \M^t} \sum_{i \in \setN^t_k} u_{ki}(z_{it} - v^*_k).
\end{equation}
Substituting $z_{\cdot t}^*$ (defined in \eqref{Primalclosedform2}) into the following \KKT system
\begin{equation*}
	\begin{aligned}
	&\frac{\partial L(z_{\cdot t}, u)}{\partial z_{it}} = \frac{-p_{it}}{\sum_{i=1}^s p_{it} z_{it} } + \sum_{k \in \M^t_{i}} u_{ki}=0,~\forall~i \in [s], \\
	& u_{ki} \geq 0,~z_{it} \leq v^*_k,~ u_{ki} (z_{it} - v^*_k) = 0,~\forall~k \in \M^t,~i \in \setN^t_k ,
	\end{aligned}
\end{equation*}
we can see that $u^*$ defined by 
\begin{equation}
	\label{defv}
	u^*_{ki} = 
	\begin{cases}
		\frac{p_{it}}{\delta_t},&~\text{if} ~k = k_i; \\[5pt]
		0,  & ~\text{otherwise},
	\end{cases}~\forall~i \in [s],~k \in \M^t_i,~\text{where}~\delta_t=\sum_{i=1}^s p_{it} v^*_{k_i},
\end{equation}
is an optimal dual solution.
Substituting  $u^*$ and $\Psi_t((v^t)^*, \eta_t^*)=\eta_t^* -\ln\left(\sum_{i=1}^s p_{it} v^*_{k_i} \right) = \eta_t^*-\ln\delta_t$ into \eqref{convexbfc} and regrouping the terms,
the Benders feasibility cut \eqref{convexbfc}  becomes
\begin{equation}
	\label{GBcut2-2}
	\sum_{k=1}^{m_t} {p(\setN^t_k \backslash  \setN^t_{[k-1]})} v_{k}  \geq {\delta}_t (\eta_t + 1 - \ln\delta_t),
\end{equation}
where $\setN^t_{[k-1]} = \bigcup_{k' \in [k-1]} \setN^t_{k'}$, $\setN^t_{[0]} = \varnothing$, and $p(\setN^t_k \backslash  \setN^t_{[k-1]})= \sum_{i \in \setN^t_k \backslash  \setN^t_{[k-1]}} p_{it}$ for $k \in \M^t$.

Next, we consider the case that $ \min_{k \in \M^t_{i}} v^*_{k}=0$ holds for all $i \in [s]$.
In this case, 
$z_{it}=0$ must hold for all $i \in [s]$ for any feasible solution $z_{\cdot t}$ of problem \eqref{subnonlinear}, and thus $\Psi_t((v^t)^*,\eta_t^*)=+\infty$ is not well-defined. 
To bypass this difficulty, we use \eqref{cons:vandz2}, \eqref{cons:vandzbound2}, and \eqref{independentEq} to derive a valid inequality cutting off point $(v^*, \eta^*)$. 
Specifically, from \cref{independentRem}, $\Psi'_t(v^t) \leq 0$, $t \in [T]$, hold for all feasible solutions $(v, \eta)$ of problem \eqref{mp2}, where for a solution $(v^*, \eta^*) \in [0,1]^m \times \mathbb{R}^T$, $\{\Psi'_t((v^t)^*)\}_{t \in [T]}$ are given by
\begin{equation}
	\label{subnonlinear2}
	\Psi'_t((v^t)^*) = \min \left\{1-\epsilon-\sum_{i=1}^s p_{it} z_{it}\,:\,  z_{it} \leq v^*_k ,~\forall~k \in \M^t,~i \in \setN^t_k \right\},~\forall~t\in [T].
\end{equation}
Using the result in \cref{subsec:BD1}, we can derive the Benders feasibility cut (corresponding to $\Psi'_t(v^t) \leq 0$)
\begin{equation}
	\label{GBcut2-3}
	\sum_{k=1}^{m_t} {p(\setN^t_k \backslash  \setN^t_{[k-1]})} v_{k}  \geq 1-\epsilon, 
\end{equation}
which cuts off point $(v^*, \eta^*) $ (as $ \min_{k \in \M^t_{i}} v^*_{k}=0$ holds for all $i \in [s]$).

For a solution $(v^*, \eta^*)\in [0,1]^m \times \mathbb{R}^T$, the Benders feasibility cuts \eqref{GBcut2-2} or \eqref{GBcut2-3} can be separated with the complexity of  $\mathcal{O}(\sum_{t\in [T]} (m_t\log  m_t  + \sum_{k \in [m_t]}|\setN^{t}_k|))$, where $\mathcal{O}(m_t\log  m_t)$ is taken by sorting $\{v^*_k\}_{k \in \M^t}$.
It is worthwhile remarking that unlike the \BD algorithm for formulation \eqref{mip} in which at most one cut is added in one iteration, we can add up to $T$  cuts in one iteration in the \BD algorithm for formulation \eqref{minlp}. 
This feature renders the \BD algorithm able to converge more quickly when solving \eqref{minlp}, especially when the number of blocks is large; see \cref{sec:comparetwoproblem} further ahead.

\section{Strong valid inequalities}
\label{sec:lci}

To further speed up the convergence of the \BnC algorithm based on Benders feasibility cuts, here we investigate the polyhedral structure $\conv(\X)$, where 
\begin{equation}\label{setX}
	\X =\left\{ v \in \{0,1\}^m \stt \Prb\{v \geq \xi\} \geq 1-\epsilon  \right\},
\end{equation}
and develop strong valid inequalities for $\conv(\X)$. 
For the case of a finite discrete distribution of $\xi$ and its special case with a block structure, $\X$ reduces to $\X^g=\proj_{v}(\Y^g)$ and $\X^b=\proj_{v}(\Y^b)$, where
\begin{align*}
		& \rev{\Y^g} = \left\{ (v,z) \in \{ 0,1\}^{m+s} \stt v_k \geq z_i,~\forall ~k \in [m], ~i \in \setS_k, ~\sum_{i=1}^s p_i z_i \geq 1-\epsilon \right\},\\
		& \rev{\Y^b} = \left\{ (v,z,\eta) \in \{ 0,1\}^{m+sT}\times \mathbb{R}^T \stt v_k \geq z_{it},~\forall ~t \in [T], ~k \in \M^t, ~i \in \setS_k^t,\vphantom{\eta_t \leq \ln\left(\sum_{i=1}^s p_{it}z_{it} \right),~\forall~t \in [T],~\sum_{i=1}^T \eta_t \geq \ln(1-\epsilon)} \right.\nonumber\\
		&  \quad\qquad\qquad\qquad\qquad\qquad\qquad \qquad\left.~\eta_t \leq \ln\left(\sum_{i=1}^s p_{it}z_{it} \right),~\forall~t \in [T],~\sum_{t=1}^T \eta_t \geq \ln(1-\epsilon) \right\}.
\end{align*}

Observe that the Benders feasibility cuts \eqref{feasibilitycut} and \eqref{GBcut2-3} are valid for linear relaxations of $\X^g$ and $\X^b$ (obtained by relaxing integrality constraints on variables $v$ in $\X^g$ and $\X^b$, respectively), and thus also valid for $\X^g$ and $\X^b$.
Similarly, summing up the Benders feasibility cuts \eqref{GBcut2-2} for all $t \in [T]$ and using $\sum_{t=1}^T \eta_t \geq \ln(1-\epsilon) $, we obtain $\sum_{t=1}^T \frac{1}{\delta_t} \sum_{k=1}^{m_t} {p(\setN^t_k \backslash  \setN^t_{[k-1]})} v_{k}  \geq T - \sum_{t=1}^T  \ln\delta_t + \ln (1-\epsilon)$, which is valid for the linear relaxation of $\X^b$ and thus also valid for $\X^b$.
Different from the Benders feasibility cuts that only lead to (potentially weak) inequalities that are valid for the linear relaxation of $\X^g$ or $\X^b$, the valid inequalities developed in this section will explicitly take the integrality constraints on variables $v$ into consideration, and are strong (facet-defining) valid inequalities for $\conv(\X^g)$ and $\conv(\X^b)$.

\subsection{Lifted cover inequalities for $\conv(\X)$}
\label{LCIX}
In this subsection, we provide a general procedure for deriving strong valid inequalities for $\conv(\X)$. 
Throughout, we assume that $\Prb\{v \geq \xi\}$ can be efficiently computed for any $v\in \{0,1\}^m$.
This assumption is reasonable for the two special cases $\X^g$ and $\X^b$ (as $\xi$ and $\{\xi^t\}_{t \in [T]}$ have finite discrete distributions, respectively).  

Given a subset $\setC \subseteq [m]$, let $v_{\setC}$ and $\xi_\setC$ be the subvectors of $v$ and $\xi$ formed by the rows in $\setC$, respectively, and define
\begin{equation}\label{lambdaC}
	\lambda(\setC)=\Prb\{\boldsymbol{0} \ngeq \xi_\setC\}.
\end{equation}
Observe that $\lambda(\setC)$ is a nondecreasing set function with respect to $\setC$, which is equal to the probability $\Prb\{v_\setC \ngeq \xi_\setC\}$ when $v_\setC = \boldsymbol{0}$, \ie $v_k=0$ holds for all $k \in \setC$.
If this probability is larger than the confidence parameter $\epsilon$, the probabilistic constraint $\Prb\{v \geq \xi\} \geq 1-\epsilon$ with $v_\setC = \boldsymbol{0}$ cannot be satisfied as $\Prb\{v \geq \xi\}\leq \Prb\{v_\setC \geq \xi_\setC\} = 1-\Prb\{v_\setC \ngeq \xi_\setC\} = 1-\Prb\{\boldsymbol{0} \ngeq \xi_\setC\} < 1- \epsilon$.
Therefore, to ensure that the probabilistic constraint $\Prb\{v \geq \xi\} \geq 1-\epsilon$ is satisfied, $v_k =0$, $k \in \setC$, cannot be simultaneously satisfied. This motivates the following definition of \emph{cover}.
\begin{definition}
	A subset $\setC\subseteq [m]$ is a cover for $\X$ if $\lambda(\setC) > \epsilon$. It is minimal if and only if $\lambda(\setC\backslash\{k\}) \leq \epsilon$ holds for all $k \in \setC$.
\end{definition}
By definition, if $\setC$ is a cover, then the cover inequality  
\begin{equation}\label{coverineq}
	\sum_{k \in\setC}v_k \geq 1
\end{equation}
is valid for $\X$.
Moreover, letting $\mathfrak{C}$ be the set of all minimal covers,
then $\X = \{ v \in \{0,1\}^m \, : \, \sum_{k \in \setC} v_k \geq 1, ~\forall~\setC \in \mathfrak{C} \}$, that is, all cover inequalities provide a linear description of $\X$.

Let $\X(\setC)=  \{ v \in \X \, : \, v_k = 1,~ \forall~k \in [m] \backslash \setC\}$, obtained by fixing $v_k=1$ for all $k \in [m]\backslash \setC$ in $\X$.
The minimal cover inequality \eqref{coverineq} is \emph{strong} with respect to $\conv(\X(\setC))$; indeed, it is facet-defining for $\conv(\X(\setC))$.
In general, however, the cover inequality \eqref{coverineq} could be weak with respect to $\conv(\X)$.
To overcome the weakness, we use the sequential lifting technique \citep{Wolsey1976,Richard2011} to convert inequality \eqref{coverineq} into a strong facet-defining inequality for $\conv(\X)$.
Sequential lifting is a procedure in which the projected variables $v_k$, $k \in [m]\backslash \setC$, are introduced into the inequality \eqref{coverineq} one at a time in some lifting order $\ell_1, \ldots, \ell_{m-|\setC|}$ (this order is also known as \emph{lifting sequence}).
More precisely, let
\begin{equation}
\label{p1dim}
 \sum_{k \in \setC}v_k + \sum_{k \in \setL_{p-1}}\alpha_k (v_k -1)\geq 1
\end{equation}
be a so-far generated inequality for $\X(\setC\cup \setL_{p-1})$, where $\alpha_{k}$'s are the lifting coefficients and $\setL_{p-1}:= \{ \ell_1,\ldots, \ell_{p-1} \}$.
To lift variable $v_{\ell_p}$, i.e., find the largest lifting coefficient $\alpha_{\ell_p}$ such that inequality
\begin{equation}
\label{mediateineq}
\sum_{k \in \setC}v_k + \alpha_{\ell_p} (v_{\ell_p}-1) + \sum_{k \in \setL_{p-1}}\alpha_k (v_k -1)  \geq 1
\end{equation}
is valid for $\X(\setC\cup \setL_{p})$, it suffices to solve the following \emph{lifting problem}
\begin{equation}\label{liftingprob}
	\alpha_{\ell_p} = \min_{v \in \{0,1\}^{|\setC\cup \setL_p|}}\left\{ \sum_{k \in \setC}v_k+ \sum_{k \in \setL_{p-1}}\alpha_k (v_k -1)  -1 \, : \, \Prb\{v_{\setC\cup \setL_p} \geq \xi_{\setC\cup \setL_p}\} \geq 1-\epsilon,~v_{\ell_p} =0  \right\}.
\end{equation}
Repeating the above lifting procedure until $ p = m - |\setC| $, we obtain {the}  \emph{lifted cover inequality}
\begin{equation}
	\label{lci}
	 \sum_{k \in \setC}v_k + \sum_{k \in [m]\backslash \setC}\alpha_k (v_k -1) \geq 1,
\end{equation}
which depends on the lifting order  $\ell_1, \ldots, \ell_{m-|\setC|}$; that is, different lifting orders of $[m]\backslash \setC$ may lead to different lifted cover inequalities.
Since the minimal cover inequality \eqref{coverineq} defines a facet of $\conv(\X(\setC))$, it follows from \cite{Wolsey1976} that the lifted cover inequality \eqref{lci} defines a facet of $\conv(\X)$.
\begin{lemma}[\cite{Wolsey1976}]
	\label{facet}
	For $1\leq p\leq m - |\setC|$, inequality \eqref{p1dim} defines a facet of the polyhedron $\conv(\X(\setC \cup \setL_{p-1}))$. 
	In particular, the lifted cover inequality \eqref{lci} defines a facet of the polyhedron $\conv(\X)$.
\end{lemma}

The above procedure lifts variables fixed at $1$ into a known inequality, called \emph{down-lifting} procedure in the literature \citep{Gu1998}.
In addition to this, we can also use the \emph{up-lifting}
 procedure, where variables are allowed to be fixed at $0$ and introduced to the inequality; see \cite{Gu1998} and \cite{Song2014,Wang2021} for similar lifting procedures arising in deterministic and probabilistic knapsack polytopes, respectively.
In particular, let $\setC$, $\SN_0$, and $\SN_1$ be a partition of $[m]$, where $\SN_0$ and $\SN_1$ are the sets of variables fixed at $0$ and $1$, respectively, and $\setC$ is a minimal cover of $\X \cap \{v\, :\, v_k = 0, ~\forall~k \in \SN_0 \}\cap \{v\, :\, v_k = 1, ~\forall~k \in \SN_1 \} $ (that is, $\lambda(\setC\cup \SN_0) > \epsilon$ and $\lambda(\setC\cup \SN_0\backslash\{k\})\leq \epsilon$ for all $k \in \setC$).
The lifted cover inequality can then be written as 
\begin{equation}\label{lci-end}
	\sum_{k \in \setC}v_k+ \sum_{k \in \SN_0} \alpha_k v_k + \sum_{k  \in \SN_1}\alpha_k (v_k -1)  \geq 1,
\end{equation}
where the lifting coefficients $\alpha_k$ for $k \in \SN_0 \cup \SN_1$ are computed in a similar manner.

Given a point $v^* \in [0,1]^m$ encountered in the \BnC algorithm (of solving the Benders reformulation of problems \eqref{mip} and \eqref{minlp}),  the separation problem of the lifted cover inequalities \eqref{lci-end} asks to either find an inequality violated by point $v^*$ or prove $v^*  \in \conv(\X)$. 
To solve the separation problem, we first apply a heuristic procedure to identify a minimal cover. 
In particular, we let $\setC:= \{k \in [m] \stt v^*_k = 0 \}$ and check whether $\setC$ is a cover; if so, we stop immediately with a cover inequality \eqref{coverineq} violated by $v^*$.
Otherwise, let $\setF := \{k \in [m] \stt 0 < v^*_k < 1\}$ and $k_1, k_2, \ldots, k_{|\setF|}$ be a permutation of $\setF$ such that $v^*_{k_1} \leq v^*_{k_2} \leq \cdots \leq v^*_{k_{|\setF|}}$.
Then for $\ell= 1, \ldots |\setF|$, we iteratively add the $v^*_{k_\ell}$ into $\setC$ until $\setC$ is a cover.
Observe that the resultant cover $\setC$ may not be minimal, so we iteratively remove elements from $\setC$ in a nonincreasing order of values $\{v_k^*\}_{k \in \setC}$ until it becomes a minimal cover.

Now, if $|\setC \cap \setF| \leq 1$, then we lift from the cover inequality $v(\setC) \geq 1$ and perform a down-lifting procedure on variables $\{v_k\}_{k \in [m]\backslash \setC}$ to obtain a lifted cover inequality \eqref{lci}.
We lift variables $\{v_k\}_{k \in [m]\backslash \setC}$ in a nondecreasing order of $\{v^*_k\}_{k \in [m]\backslash \setC}$, since the earlier a variable is lifted, the larger the lifting coefficient will be.
If $|\setC \cap \setF| \geq 2$, we follow \cite{Gu1998,Song2014,Wang2021} to use a three-phase lifting procedure to obtain a lifted cover inequality.
Specifically,  letting $\setC_1 := \{k \in \setC \stt v^*_k > 0 \}$, $\setC_2 := \{k \in \setC \stt v^*_k = 0 \}$, and $\SN_1 =  \{k \in [m] \stt v^*_k = 1\}$, then we start from the cover inequality $v(\setC_1) \geq 1$, perform (i) a down-lifting procedure on  variables $\{v_k\}_{k \in [m] \backslash (\setC \cup \SN_1)}$, (ii) an up-lifting procedure  on variables $\{v_k\}_{k \in \setC_2}$, and (iii)  a down-lifting procedure on variables $\{v_k\}_{k \in \SN_1}$.
Similarly, when lifting variables $\{v_k\}_{k \in [m] \backslash (\setC \cup \SN_1)}$, we use a nondecreasing order of $\{v^*_k\}_{k \in [m]\backslash \setC}$, as to obtain a violated lifted cover inequality.
However, as the lifting of variables $\{v_k\}_{k \in \setC_2}$ and $\{v_k\}_{k \in \SN_1}$ does not contribute to the violation of the lifted cover inequality, we simply use a variable index order instead. 
This also implies that we can abandon the lifting procedure if $v(\setC_1) + \sum_{k \in [m] \backslash (\setC \cup \SN_1)}\alpha_{k} (v_k-1) \geq 1$ is not violated by $v^*$, thereby avoiding further unnecessary computational effort.

\subsection{Lifted cover inequalities for $\conv(\X^g)$}
\label{lciigeneral}

Letting $\setS_{\setC} = \bigcup_{k \in \setC} \setS_k$, 
then for $\X^g$, $\lambda(\setC)$ defined in \eqref{lambdaC} reduces to
\begin{equation}\label{coverdef}
	\lambda(\setC) =\Prb\{\boldsymbol{0} \ngeq \xi_\setC\} = \sum_{i \in [s]} p_i \chi(\boldsymbol{0} \ngeq \xi_{i,\setC}) =  \sum_{i \in \setS_\setC} p_i  = p(\setS_\setC),
\end{equation}
where $\chi(\cdot)$ is an indicator function.
Thus, a subset $\setC\subseteq [m]$ is a cover for $\X^g$ if $p(\setS_{\setC}) > \epsilon$, and it is minimal if and only if $p(\setS_{\setC\backslash \{k\}}) \leq \epsilon$ holds for all $k \in \setC$.
In addition, the lifting problem in \eqref{liftingprob} reduces to
\begin{subequations}
	\label{liftingprobg}
	\begin{align}
		{\alpha}_{\ell_p}= \min \quad & \sum_{k \in \setC}v_k  + \sum_{k \in \setL_{p-1}}\alpha_k (v_k -1) -1\\
		{\text{s.t.}}
		\quad &v_k \geq z_i,~\forall~k \in \setC\cup\setL_{p-1},~i\in \setN_k, \label{liftP1}\\
		\quad & \sum_{i=1}^s p_i z_i \geq 1 -  \epsilon ,~z_i=0,~\forall~i\in  \setS_{\ell_p},\label{liftP2}\\
		\quad & v_{\setC \cup \setL_{p-1}} \in\{0,1\}^{|\setC \cup \setL_{p-1}|},~ z\in\{0,1\}^s.
	\end{align}
\end{subequations}
Problem \eqref{liftingprobg} is a precedence constrained knapsack problem, which is strongly NP-hard \citep{Johnson1983}\rev{, making it challenging to design an empirically efficient algorithm (e.g., a pseudo-polynomial time algorithm) to solve it.}
To resolve the difficulty, we follow \cite{Espinoza2015} to solve its \LP relaxation
\begin{equation}
	\label{liftingprobgU}
	\begin{aligned}
		\hat{\alpha}_{\ell_p}= \min \quad &\sum_{k \in \setC}v_k +\sum_{k \in \setL_{p-1}}\alpha_k (v_k -1)  -1\\
		{\text{s.t.}}\quad & \eqref{liftP1},~\eqref{liftP2}, ~v_{\setC \cup \setL_{p-1}} \in[0,1]^{|\setC \cup \setL_{p-1}|},~ z\in[0,1]^s,
	\end{aligned}
\end{equation}
to derive a lower bound $\hat{\alpha}_{\ell_p}$ for the lifting coefficient ${\alpha}_{\ell_p}$, thereby also obtaining a valid lifted  cover inequality for $\conv(\X^g)$.
Moreover, the above approximation lifting problem \eqref{liftingprobgU} can be solved using the customized algorithm of \cite{Chicoisne2012} with the complexity of $\mathcal{O}(n \gamma\log n)$, where $n=|\setC\cup \setL_{p-1}| + s$ is the number of variables in \eqref{liftingprobgU} and $\gamma=\sum_{k \in \setC\cup \setL_{p-1}} |\setS_k|$ is the number of constraints in \eqref{liftP1}.
Based on our computational experience, this customized algorithm is much more efficient than the direct use of an \LP solver. 

Note that in contrast to the optimal value of the exact lifting problem \eqref{liftingprobg}, which is nonnegative and integral, the optimal value of the approximation lifting problem  \eqref{liftingprobgU}, however, could be negative and fractional.
To overcome this weakness and obtain a stronger approximation lifting coefficient, we can use $\max\{ \lceil \hat{\alpha}_{\ell_p} \rceil , 0 \}$ as the lifting coefficient instead of $\hat{\alpha}_{\ell_p} $.

\subsection{Lifted cover inequalities for $\conv(\X^b)$}\label{lcispecial}

Given a subset $\setC \subseteq [m]$, let $\setC^t:= \setC \cap \M^t $, $\setS_{\setC^t} =  \bigcup_{k \in \setC^t} \setS^t_k$, and $p_{\cdot t}(\setS_{\setC^t}):=
\sum_{i \in \setS_{\setC^t}} p_{it}$, where $t \in [T]$.
Similar to \eqref{coverdef}, $p_{\cdot t}(\setS_{\setC^t}):=
\sum_{i \in \setS_{\setC^t}} p_{it}$ denotes  the probability $\Prb\{\boldsymbol{0} \ngeq \xi_{\setC^t}\}$.
Thus, $\lambda(\setC)$ defined in \eqref{lambdaC} reduces to 
\begin{equation}\label{ineqC}
	\lambda(\setC) =\Prb\{\boldsymbol{0} \ngeq \xi_\setC\} \overset{(a)}{=} \prod_{t = 1}^T \Prb\{\boldsymbol{0} \ngeq \xi_{\setC^t}\} =  
	\prod_{t = 1}^T \sum_{i \in \setS_{\setC^t}} p_{it}=\prod_{t = 1}^T p_{\cdot t}(\setS_{\setC^t}),
\end{equation}
where (a) follows from the fact that the random vectors $\xi_{\setC^1}, \ldots, \xi_{\setC^T}$ are independent 
(as $\xi^1, \ldots, \xi^T$ are independent and  $\xi_{\setC^1}, \ldots, \xi_{\setC^T}$ are subvectors of $\xi^1, \ldots, \xi^T$).  
Thus, we can use \eqref{ineqC} to determine whether $\setC$ is a (minimal) cover. 
In addition, letting $t_0 \in [T]$ be such that $\ell_p \in \M^{t_0}$, then the lifting problem in \eqref{liftingprob} reduces to 
\begin{subequations}\label{liftingprobb}
	\begin{align}
		\alpha_{\ell_p}= \min \quad & \sum_{k \in \setC}v_k+ \sum_{k \in \setL_{p-1}}\alpha_k (v_k -1) -1\\
		{\text{s.t.}}
		\quad &v_k \geq z_{it},~\forall~t \in [T],~k \in \setC\cup \setL_{p-1},~i\in \setN_k^t, \label{liftbP1}\\
		\quad & ~\eta_t \leq \ln\left(\sum_{i=1}^s p_{it}z_{it} \right),~\forall~t \in [T],~\sum_{t=1}^T \eta_t \geq \ln(1-\epsilon) \label{liftbP2},\\
		\quad & v_{\setC \cup \setL_{p-1}} \in\{0,1\}^{|\setC \cup \setL_{p-1}|},~ z\in\{0,1\}^{s\times T},~z_{i t_0}=0,~\forall~i \in \setS_{\ell_p}^{t_0}.\label{liftbP3}
	\end{align}
\end{subequations}

Note that different from problem \eqref{liftingprobg} which is an \MIP problem, the lifting problem \eqref{liftingprobb} is \rev{a relatively difficult \MINLP} problem as it includes the  nonlinear constraints $\eta_t \leq \ln\left(\sum_{i=1}^s p_{it}z_{it}\right)$ in \eqref{liftbP2}.
In order to avoid the nonlinearity in \eqref{liftingprobb}, observe that
$$
	\ln(1-\epsilon) \leq \sum_{t=1}^T \eta_t \leq \sum_{t=1}^T \ln \left(\sum_{i=1}^s p_{it}z_{it} \right)\overset{(a)}{\leq} \sum_{t=1}^T \left(\sum_{i=1}^s p_{it}z_{it}-1\right),
$$
where (a) follows from the fact that $\ln(\sigma) \leq \sigma - 1$ holds for any $\sigma \in (0,1]$.
Thus,
\begin{equation}\label{knaps}
	\sum_{t=1}^T \sum_{i =1}^s p_{it}z_{it} \geq T + \ln (1-\epsilon)
\end{equation}
is a valid (linear) inequality for problem \eqref{liftingprobb}.
Substituting the constraints in \eqref{liftbP2} by \eqref{knaps}, we obtain an \MIP relaxation of problem \eqref{liftingprobb}:
\begin{equation}\label{liftingprobbU}
	\begin{aligned}
		\bar{\alpha}_{\ell_p}= \min \quad & \sum_{k \in \setC}v_k+ \sum_{k \in \setL_{p-1}}\alpha_k (v_k -1) -1\\
		{\text{s.t.}}
		\quad &\eqref{liftbP1}, ~\eqref{liftbP3}, ~\eqref{knaps}.
	\end{aligned}
\end{equation}
Observe that $\bar{\alpha}_{\ell_p} \leq \alpha_{\ell_p}$ holds, and thus solving problem \eqref{liftingprobbU} will obtain a valid lifting coefficient.
Now, similar to problem \eqref{liftingprobg}, problem \eqref{liftingprobbU} is also a precedence constrained knapsack problem.
Therefore, we can leverage the results in \cref{lciigeneral} to further derive valid lifting coefficients, thereby obtaining valid lifted cover inequalities for \rev{$\conv(\X^b)$}.

\subsection{Clique-based inequalities}
\label{subsec:clique}

As shown in \cref{LCIX}, in order to obtain a lifted cover inequality \eqref{lci-end} from a minimal cover inequality \eqref{coverineq}, one has to solve $m - |\setC|$ lifting problems of the form \eqref{liftingprob}, which could be computationally demanding (even though we can solve the \LP relaxations to obtain valid lifting coefficients when $\xi$ has a finite discrete distribution with or without a block structure).
To further speed up the computational procedure, in this section, we will present a subclass of the lifted cover inequalities, which can be constructed without the need for the \rev{potentially computationally} expensive lifting procedure\rev{, making them particularly suitable to be embedded in a \BnC framework to solve the \CCSCP.}

To proceed, we consider a minimal cover $\setC$ with size $|\setC|=2$, which we call two-cover. 
The number of two-covers is upper bounded by $\mathcal{O}(m^2)$.
Therefore, by checking whether $\lambda(\{j,k\})>\epsilon$ holds for all pairs $j,k\in [m]$ with $j \neq k$
(where $\lambda(\cdot)$ is defined in \eqref{lambdaC}), we can identify all two-covers and construct the \emph{two-cover induced graph} as follows.
\begin{definition}
	A two-cover induced graph $G =([m], \mathcal{E})$ contains the edge $\{j,k\}$ if and only if $\lambda(\{j,k\}) > \epsilon$.
\end{definition}

Letting $\setC = \{j,k\}$, then the cover inequality \eqref{coverineq} reduces to $v_j + v_k \geq 1$.
Two-cover inequalities can be strengthened using the concept of \emph{cliques}. 
More specifically, a clique with vertex set $\CQ$ is a fully  connected subgraph in $G$; that is, for any pair $j \in \CQ$ and $k \in \CQ \backslash \{j\}$, it follows $\{j,k\} \in \mathcal{E}$. 
A maximal clique is a clique in graph $G$ that cannot be extended by including any more vertices. 
For a maximal clique with vertex set $\CQ$, since at most one variable $v_k$, $k \in \CQ$, can take the value of zero, we obtain the maximal clique inequality 
\begin{equation}\label{maxclique}
	\sum_{k \in \CQ} v_k \geq |\CQ|-1,
\end{equation}
which is valid for $\X$ and stronger than the two-cover inequalities $v_j + v_k \geq 1$, $j, k \in \CQ$.

For any maximal clique with vertex set $\CQ$, the two-cover inequality $v_j + v_k \geq 1$ with $j, k \in \CQ$ is valid for $\X(\{j,k\})$.
By first lifting variables $ \{v_{k'} \, : \, {k'} \in \CQ\backslash\{j,k\}\}$ (with any lifting order of $\CQ\backslash\{j,k\}$) and then lifting variables $ \{v_{k'} \, : \, {k'} \in [m]\backslash \CQ\}$ (with any lifting order of $[m]\backslash \CQ$), we can also obtain the maximal clique inequality \eqref{maxclique}.
As a result, the maximal clique inequalities are special cases of the lifted cover inequalities in \eqref{lci-end}. 
Therefore, it immediately follows from \cref{facet} that
\begin{corollary}
	The maximal clique inequalities are facet-defining for $\conv(\X)$.
\end{corollary}

Although the maximal clique inequalities are special cases of lifted cover inequalities, they are very attractive from a computational perspective due to the following reasons.
First, compared with the lifted cover inequalities whose computation requires to solve $m-|\setC|$ lifting problems of the form \eqref{liftingprob} (or their \LP relaxations), \emph{the maximal clique inequalities admit a much more efficient computational procedure}---one just needs to extend a clique in a graph until no more vertices can be included with the complexity of $\mathcal{O}(m^2)$.
Second, given a solution $v^* \in [0,1]^m$ encountered in the \BnC algorithm (of solving the Benders reformulation of problems \eqref{mip} and \eqref{minlp}),
the problem of finding a violated maximal clique inequality is equivalent to solving a maximum weighted clique problem \citep{Atamturk2000,Achterberg2007}. 
Although this problem is NP-hard \citep[Theorem 9.2.9]{Grotschel1993}, there are various effective heuristic algorithms for finding violated inequalities in the literature; see \cite{Atamturk2000,Achterberg2007,Burke2012,Rebennack2012,Brito2021} among many of them.  
Moreover, state-of-the-art \MIP solvers have implemented effective algorithms to find violated maximal clique inequalities, which allows us to simply state the two-cover inequalities as constraints and then facilitate the powerful \MIP solvers' internal separation procedure for finding violated clique inequalities.

\section{Computational results}\label{section:num}

In this section, we present computational results to illustrate the effectiveness of the proposed \BD algorithms for solving  \CCSCPs based on formulations \eqref{mip} and \eqref{minlp}.
In particular, 
{we first evaluate} the performance effect of the proposed lifted cover inequalities for the proposed \BD algorithm. 
Then, we present computational results to show the effectiveness of the proposed \BD algorithms over state-of-the-art approaches for solving problems \eqref{mip} and \eqref{minlp}.
Finally, we compare the performance of the proposed \BD algorithms for solving problems \eqref{mip} and \eqref{minlp} when applying them as 
\rev{sample approximation problems} to solve \CCSCPs.
\rev{In Online Appendix C, we further present the computational results to analyze the effect of the scenario size on the quality of the solution returned by solving the sample approximation problem.}

\subsection{Implementation}\label{sec:implement}

The proposed \BD algorithms were implemented in Julia 1.7.3 using \cplex 20.1.0. 
We use a branch-and-Benders-cut (\BnBC) approach \citep{Rahmaniani2017} to implement the proposed \BD algorithms for solving the \CCSCPs, where the Benders feasibility cuts are separated on the fly at the nodes of the search tree.
In addition, when using the lifted cover inequalities, we separate them at the root node of the search tree.
This approach can be implemented via the cut callback framework available in modern general-purpose \MIP solvers
and has been widely applied to implement the \BD algorithm for various problems; see \cite{Gendron2016,Fischetti2016,Bodur2017,Cordeau2019} among many of them.
Before starting the \BnBC algorithm, we apply the preprocessing technique investigated in \cite{Lejeune2008,Lejeune2010,Luedtke2008} to simplify formulations \eqref{mip} and \eqref{minlp}, and add all two-cover inequalities in \cref{subsec:clique} as constraints to initialize the relaxed master problem.
Note that with the two-cover inequalities in the relaxed master problem, \cplex can automatically strengthen the two-cover inequalities into a subset of maximal clique inequalities \eqref{maxclique} and separate the remaining maximal clique inequalities on the fly.

To further speed up the \BnBC algorithm, we follow \cite{Fischetti2016} to implement a \emph{restart} mechanism to
increase the information available (i.e., effective cuts and heuristic solutions) at the root node, so that more variable fixings and generation of internal cuts will be performed at the root node.
Specifically, before entering into the final \BnC process, we terminate the algorithm immediately after solving the root node (through setting \texttt{CPX\_PARAM\_NODELIM=1}), state the generated Benders feasibility cuts and lifted cover inequalities, stored in our own data structure, as static constraints in the relaxed master problem, and update the incumbent solution. 
We repeat the above process twice and then go to the final \BnC process.

In our experiments, the parameters of \cplex were set to run the code in a single-threaded mode, with a time limit of 7200 seconds and a relative \MIP gap of 0\%. 
Unless otherwise stated, all other parameters in \cplex were set to their default values. 
The computational experiments were conducted on a cluster of computers equipped with Intel(R) Xeon(R) Gold 6140 CPU @ 2.30GHz.
\red{All code and data of the experiments are available at  \cite{LvChenChenDai2025}.}

\subsection{Testset of \CCSCP instances}\label{sec:testset}
Following \cite{Saxena2010}, we construct \CCSCP instances using the 60 deterministic \SCP instances \citep{Beasley1990}\rev{; the numbers of rows and columns $(m,n)$ of matrix $A$ in these instances are $(50,500), (200,1000), (200,2000)$, $(300,3000),(400,4000),(500,5000)$ with the corresponding numbers of instances being $5, 15,10,10$, $10,10$, respectively.} 
Specifically, we construct instances with two block sizes, namely 10 and 20,
and instances that do not have a block structure of the random variable $\xi $ (that is, $\xi$ has only a single block).
We use two probability distributions of $\xi$ called \emph{circular} and \emph{star} distributions \citep{Beraldi2002,Saxena2010} to construct scenarios of $\{\xi_i\}$ or $\{\xi_i^t\}$. 
For the circular distribution, the $k$-th entry of the random vector $\xi$ is set to $\max\{Y_{k}, Y_{(k \bmod m)+1}\}$, where $\{Y_k\}_{k \in [m]}$ are independent Bernoulli random variables following the distribution $\Pr\{Y_k = 1\}  = \alpha_k$, $\alpha_k$ are uniformly chosen from $[0.01, 0.0275]$, and
$\bmod$ denotes the modulo operator.
For the star distribution, the $k$-th entry of the random vector $\xi$ is set to $0$ if $V_k \leq 1$ and $1$ otherwise, where $V_k = Y_k + Y_{m+1}$,  $\{Y_k\}_{k \in [m+1]}$ are independent Poisson random variables following the distribution $\Pr(Y_k = \tau)= \frac{e^{-\lambda_k}(\lambda_k)^\tau}{\tau!}$ ($\tau\in \mathbb{Z}_+$), and $\lambda_k$ are uniformly chosen from $[0.1, 0.2]$.
The distributions of the random variables $\xi^t$ can be described in a similar manner.
In our random construction procedure, the scenario size $s$ and the confidence parameter $\epsilon$ are chosen from \rev{$\{10^3, 3\!\times\!10^3, 5\!\times\!10^3, 10^4, 10^6\}$} and $\{0.05,  0.1\}$, respectively.
In total, there are \rev{$3600$} \CCSCP instances of problems \eqref{mip} and \eqref{minlp}, respectively, in our testset.

\subsection{Performance effect of lifted cover inequalities}\label{sec:inequalities}

In this subsection, we evaluate the performance effect of the lifted cover inequalities in \cref{sec:lci} for solving \CCSCPs with and without the consideration of a block structure of $\xi$ (i.e., problems \eqref{minlp} and \eqref{mip}).
Let us first consider the subclass of them: the maximal clique  inequalities \eqref{maxclique}, 
which do not need to invoke the potentially \rev{computationally} expensive lifting procedure; see \cref{subsec:clique}.
To do this, we compare the following two settings:
\begin{itemize}
	\item[$\bullet$] \tbBnBC: problem \eqref{mip} or \eqref{minlp} is solved using the proposed \BD algorithm in which 
		the two-cover inequalities initialize the relaxed master problem and 
		the Benders feasibility cuts are separated on the fly at the nodes of the search tree (as mentioned earlier, \cplex can automatically strengthen the two-cover inequalities into maximal clique inequalities \eqref{maxclique});
		\item[$\bullet$] \tbBnBCC: problem \eqref{mip} or \eqref{minlp} is solved using \tbBnBC without the two-cover inequalities.
\end{itemize}

\cref{mip:clique,minlp:clique} summarize the computational results.
For each setting, we report the total number of instances solved within the time limit (\tbS), the average \CPU time in seconds (\tbT), the average number of explored nodes (\tbN), and the average percentage of gap improvement defined by $\tbGC = \frac{z_{\text{root}} - z_{\text{LP}}}{z-z_{\text{LP}}} \times 100 \%$, where $z$ is the objective value of the optimal solution or the best incumbent of problem \eqref{mip} or \eqref{minlp},
$z_{\text{LP}}$ is the optimal value of its \LP relaxation, and $z_{\text{root}}$ is the \LP relaxation bound obtained at the root node.
For instances that cannot be solved within the time limit, we report the average relative end gap (\tbG) in percentage returned by \cplex.
Under setting \tbBnBC, we additionally report the average number of the two-cover inequalities (\tbNCI). 
Throughout the paper, all averages are reported as the shifted geometric mean with a shift of $1$; see \cite{Achterberg2007}.

\begin{table}[t]
	\centering
	\small
	\renewcommand{\arraystretch}{1.2}
	\addtolength{\tabcolsep}{-2pt}
	\caption{\small Performance comparison of \tbBnBCC and \tbBnBC on problem \eqref{mip}.}\label{mip:clique}
	\begin{tabular}{llcccccccccccc}
		\toprule
		&     &  \texttt{All}  &  \multicolumn{5}{c}{\tbBnBCC}                  & \multicolumn{5}{c}{\tbBnBC}         \\
		\cmidrule(r){4-8} \cmidrule(r){9-14}   
		&     &                 &    \tbS & \tbT & \tbG &\tbN &\tbGC       & \tbS & \tbT & \tbG & \tbN & \tbGC & \tbNCI    \\
		\hline
		$\epsilon$            & 0.05      & 1800   & 1530 &    90.2 &      19.1 &     59 &      15.3       & \textbf{1635} & \textbf{59.9} & \textbf{13.7} & \textbf{24} & \textbf{25.2} &     27571 \\
		& 0.1       & 1800   & \textbf{1404} & \textbf{201.9} &      \textbf{15.0} & \textbf{315} &       \textbf{5.9}       & 1400 &   202.3 & \textbf{15.0} &    316 & \textbf{5.9} &        11 \\
		\hline
		\texttt{$(m,n)$}       
		& $(50,500)$  &  300   & \textbf{300} &    23.3 & \textbf{0.0} &    325 &       0.7       & \textbf{300} & \textbf{20.3} & \textbf{0.0} & \textbf{222} & \textbf{1.8} &        51 \\
		& $(200,1000)$  &  900   & \textbf{900} &    27.4 & \textbf{0.0} &     19 &      24.6       & \textbf{900} & \textbf{24.8} & \textbf{0.0} & \textbf{13} & \textbf{30.6} &       422 \\
		& $(200,2000)$  &  600   & \textbf{600} &    30.7 & \textbf{0.0} &     18 &      35.7       & \textbf{600} & \textbf{27.0} & \textbf{0.0} & \textbf{11} & \textbf{41.2} &       422 \\
		& $(300,3000)$  &  600   & \textbf{600} &    79.7 & \textbf{0.0} &    545 &       9.0       & \textbf{600} & \textbf{70.8} & \textbf{0.0} & \textbf{399} & \textbf{11.7} &       741 \\
		& $(400,4000)$  &  600   & \textbf{504} &   476.7 &       \textbf{4.4} &   3031 &      10.5       &  500 & \textbf{320.7} & \textbf{4.4} & \textbf{1275} & \textbf{13.0} &      1324 \\
		& $(500,5000)$  &  600   &   30 &  6970.7 &      20.4 &  58685 &       1.0       & \textbf{135} & \textbf{4859.2} & \textbf{18.1} & \textbf{23908} & \textbf{1.6} &      1417 \\
		\bottomrule
	\end{tabular}
\end{table}

\begin{table}[t]
	\centering
	\small
	\renewcommand{\arraystretch}{1.2}
	\addtolength{\tabcolsep}{-2pt}
	\caption{\small Performance comparison of \tbBnBCC and \tbBnBC on problem \eqref{minlp}.}\label{minlp:clique}
	\begin{tabular}{llcccccccccccc}
		\toprule
		&     &  \texttt{All}  &  \multicolumn{5}{c}{\tbBnBCC}                  & \multicolumn{5}{c}{\tbBnBC}         \\
		\cmidrule(r){4-8} \cmidrule(r){9-14}   
		&     &                 &    \tbS & \tbT & \tbG &\tbN &\tbGC       & \tbS & \tbT & \tbG & \tbN & \tbGC & \tbNCI    \\
		\hline
		$\epsilon$            & 0.05      & 1800   & 1644 &    62.7 &      18.2 &     69 &      17.3       & \textbf{1701} & \textbf{50.7} & \textbf{14.1} & \textbf{38} & \textbf{27.2} &     27643 \\
		& 0.1       & 1800   & 1540 &    93.5 &      18.1 & \textbf{90} &      10.5       & \textbf{1555} & \textbf{90.6} & \textbf{14.6} &     91 & \textbf{11.8} &        11 \\
		\hline
		\texttt{$(m,n)$}       & $(50,500)$  &  300   & \textbf{300} &    18.5 & \textbf{0.0} &    172 &       1.2       & \textbf{300} & \textbf{17.6} & \textbf{0.0} & \textbf{155} & \textbf{2.1} &        52 \\
		& $(200,1000)$  &  900   & \textbf{900} &    20.8 & \textbf{0.0} &      7 &      40.6       & \textbf{900} & \textbf{20.6} & \textbf{0.0} & \textbf{5} & \textbf{44.6} &       423 \\
		& $(200,2000)$  &  600   & \textbf{600} &    22.3 & \textbf{0.0} &      4 &      55.9       & \textbf{600} & \textbf{21.1} & \textbf{0.0} & \textbf{3} & \textbf{59.5} &       423 \\
		& $(300,3000)$  &  600   &  593 &    51.1 &       1.3 &    203 &      13.2       & \textbf{597} & \textbf{45.0} & \textbf{0.7} & \textbf{147} & \textbf{16.7} &       737 \\
		& $(400,4000)$  &  600   &  506 &   133.9 &       9.0 &    623 &       7.2       & \textbf{525} & \textbf{91.2} & \textbf{5.9} & \textbf{307} & \textbf{13.8} &      1340 \\
		& $(500,5000)$  &  600   &  285 &  2959.0 &      22.2 &  50031 &       2.5       & \textbf{334} & \textbf{2618.7} & \textbf{18.5} & \textbf{45984} & \textbf{4.4} &      1422 \\
		\bottomrule
	\end{tabular}
\end{table}

We can observe from \cref{mip:clique,minlp:clique} that  for instances with $\epsilon = 0.1$,  the number of added two-cover inequalities is very small; see column \tbNCI.
This is because when $\epsilon$ is large, the condition $\lambda(\{j,k\}) \geq \epsilon$ is likely not to be satisfied, and hence only a small number of two-cover inequalities are constructed.
Due to this, the performance effect of the two-cover inequalities (or their strengthened version, maximal clique inequalities) on the proposed \BD algorithm is neutral on instances with $\epsilon = 0.1$. 
In contrast, for instances with $\epsilon = 0.05$ which is relatively small, 
a large number of two-cover inequalities can be derived from problems \eqref{mip} and \eqref{minlp}.
Due to this, \cplex can construct more clique inequalities, which is evidenced by column \tbGC where we observe that the gap improvement returned by \tbBnBC is much better than that returned by \tbBnBCC.
Due to the improvement on the \LP relaxation bound, the performance of \tbBnBC is much better than that of \tbBnBCC.
Overall, for problem \eqref{mip}, \tbBnBC can solve \rev{$1635$} instances among the \rev{$1800$} instances to optimality, whereas \tbBnBCC can only solve \rev{$1530$} of them to optimality, and the average \CPU time and number of explored nodes are reduced by factors of \rev{$1.5$} and \rev{$2.5$}, respectively;
for problem \eqref{minlp},  \tbBnBC can solve \rev{$1701$} instances to optimality, whereas \tbBnBCC can only solve \rev{$1644$} of them to optimality, and the average \CPU time and number of explored nodes are reduced by factors of \rev{$1.2$} and $1.8$, respectively.
\rev{When comparing the results grouped by the numbers of rows and columns $ (m,n)$, we see that with the increasing of $(m,n)$,
solving problems \eqref{mip} and \eqref{minlp} becomes more difficult and the maximal clique inequalities become more effective.
In particular, for the case with $(m,n)=(500,5000)$, the proposed maximal clique inequalities enable to solve $105$ and $49$ more instances of problems \eqref{mip} and \eqref{minlp}, respectively, to optimality within the time limit of $7200$ seconds.}

Next, we evaluate the performance effect of the (general) lifted cover inequalities \eqref{lci-end} in \rev{\cref{lciigeneral,lcispecial}}, which, compared with the maximal clique inequalities, require to solve the \LP relaxations of the lifting problems \eqref{liftingprobg} and \eqref{liftingprobbU} using the customized algorithm of \citet{Chicoisne2012}.
To do this, we compare the performance of \tbBnBC, where the maximal clique inequalities \eqref{maxclique} are added,  with
\begin{itemize}
	\item[$\bullet$] \tbBnBCCL: problem \eqref{mip} or \eqref{minlp} is solved using \tbBnBC with the lifted cover inequalities in \cref{lciigeneral} or \cref{lcispecial}, respectively.
\end{itemize}
\cref{mip:lci,minlp:lci} report the computational results of \tbBnBC and \tbBnBCCL on  instances \rev{with $10^3,~3\!\times\!10^3,~5\!\times\!10^3$, and $10^4$ scenarios.}
We do not report the computational results on instances with $10^6$ scenarios as solving the \LP relaxation of the lifting problem \eqref{liftingprobg} or \eqref{liftingprobbU} is too time-consuming for the customized algorithm of \citet{Chicoisne2012}.
Under setting \tbBnBCCL, we report the average \CPU time spent in the \rev{computation} of the lifted cover inequalities in seconds (\tbST) and the average number of added lifted cover inequalities (\tbNLCI).\looseness=-1

From \cref{minlp:lci}, we observe from column \tbNLCI that for problem \eqref{minlp},  only a small number of lifted cover inequalities are added, resulting in a small improvement on the \LP relaxation bound.
Consequently, the number of explored nodes does not significantly decrease.
In contrast, for problem \eqref{mip}, \cref{mip:lci} shows that a large number of lifted cover inequalities can be derived, especially for instances with $\epsilon = 0.1$.
As a result, compared with \tbBnBC, \tbBnBCCL achieves a much better gap improvement and a much smaller number of explored nodes; 
the number of explored nodes is reduced by factors $1.2$ and \rev{$3.4$} for instances with $\epsilon = 0.05$ and $\epsilon = 0.1$, respectively.
Unfortunately, for both problems \eqref{mip} and \eqref{minlp}, solving the \LP relaxations of the lifting problems \eqref{liftingprobg} and \eqref{liftingprobbU} is very computationally demanding,
\rev{especially for instances with a large $m$ (as the problem sizes of the lifting problems are large).}
Thus, the (general) lifted cover inequalities \eqref{lci-end} cannot improve the overall performance of \tbBnBC.
Nevertheless, for problem \eqref{mip}, the addition of the lifted cover inequalities \eqref{lci-end} enables to solve \rev{$28$} more instances with \rev{$(m,n)=(400,4000)~\text{or}~(500,5000)$} to optimality, suggesting that the lifted cover inequalities can be used to tackle challenging instances of problem \eqref{mip} \rev{(with large numbers of rows and columns)}.

\begin{table}[t]
	\centering
	\small
	\renewcommand{\arraystretch}{1.2}
	\addtolength{\tabcolsep}{-3pt}
	\caption{\small Performance comparison of \tbBnBC and \tbBnBCCL on problem \eqref{mip}.}\label{mip:lci}
	\begin{tabular}{llccccccccccccc}
		\toprule
		&     &  \texttt{All}  &    \multicolumn{5}{c}{\tbBnBC}                         & \multicolumn{7}{c}{\tbBnBCCL}         \\
		\cmidrule(r){4-8} \cmidrule(r){9-15}   
		&     &                 &    \tbS & \tbT & \tbG &\tbN &\tbGC     & \tbS & \tbT & \tbG & \tbN & \tbGC & \tbST  & \tbNLCI \\
		\hline
		$\epsilon$            & 0.05      & 1440   & 1309 & \textbf{56.0} &      14.2 &     41 &      24.2      & \textbf{1320} &   134.8 & \textbf{14.0} & \textbf{33} & \textbf{30.2} &   48.6  &     13  \\
		& 0.1       & 1440   & 1123 & \textbf{185.9} &      16.5 &    311 &       5.7      & \textbf{1140} &   963.7 & \textbf{16.2} & \textbf{91} & \textbf{17.6} &  551.3  &    194  \\
		\hline
		\texttt{$(m,n)$}       
		& $(50,500)$  &  240   & \textbf{240} & \textbf{19.4} & \textbf{0.0} &    220 &       1.8      & \textbf{240} &    38.6 & \textbf{0.0} & \textbf{176} & \textbf{7.6} &   19.3  &     27  \\
		& $(200,1000)$  &  720   & \textbf{720} & \textbf{22.4} & \textbf{0.0} &     13 &      28.9      & \textbf{720} &   116.0 & \textbf{0.0} & \textbf{6} & \textbf{51.7} &   71.3  &     32  \\
		& $(200,2000)$  &  480   & \textbf{480} & \textbf{24.5} & \textbf{0.0} &     11 &      41.2      & \textbf{480} &   132.3 & \textbf{0.0} & \textbf{3} & \textbf{73.9} &   83.0  &     38  \\
		& $(300,3000)$  &  480   & \textbf{480} & \textbf{63.1} & \textbf{0.0} &    402 &      10.9      & \textbf{480} &   441.2 & \textbf{0.0} & \textbf{203} & \textbf{21.8} &  354.1  &     89  \\
		& $(400,4000)$  &  480   &  403 & \textbf{298.3} &       4.8 &   1424 &      12.7      & \textbf{420} &   890.8 & \textbf{4.4} & \textbf{773} & \textbf{24.2} &  493.0  &     92  \\
		& $(500,5000)$  &  480   &  109 & \textbf{4822.0} &      18.5 &  34981 &       1.6      & \textbf{120} &  5372.8 & \textbf{18.3} & \textbf{31040} & \textbf{3.0} &  501.8  &     66  \\
		\bottomrule
	\end{tabular}
\end{table}

\begin{table}[t]
	\centering
	\small
	\renewcommand{\arraystretch}{1.2}
	\addtolength{\tabcolsep}{-3pt}
	\caption{\small Performance comparison of \tbBnBC and \tbBnBCCL on problem \eqref{minlp}.}\label{minlp:lci}
	\begin{tabular}{llccccccccccccc}
		\toprule
		&     &  \texttt{All}  &    \multicolumn{5}{c}{\tbBnBC}                         & \multicolumn{7}{c}{\tbBnBCCL}         \\
		\cmidrule(r){4-8} \cmidrule(r){9-15}   
		&     &                 &    \tbS & \tbT & \tbG &\tbN &\tbGC     & \tbS & \tbT & \tbG & \tbN & \tbGC & \tbST  & \tbNLCI \\
		\hline
		$\epsilon$            & 0.05      & 1440   & \textbf{1373} & \textbf{42.9} & \textbf{13.5} &     53 &      26.8      & 1372 &    48.9 &      13.7 & \textbf{52} & \textbf{26.9} &    1.9  &     $<$1  \\
		& 0.1       & 1440   & 1261 & \textbf{75.6} &      16.8 &    100 &      11.9      & \textbf{1271} &   182.0 & \textbf{13.8} & \textbf{95} & \textbf{13.6} &   40.8  &      3  \\
		\hline
		$(m,n)$       
		& $(50,500)$  &  240   & \textbf{240} & \textbf{16.5} & \textbf{0.0} & \textbf{156} &       1.9      & \textbf{240} &    17.6 & \textbf{0.0} &    158 & \textbf{1.9} &    1.7  &     $<$1  \\
		& $(200,1000)$  &  720   & \textbf{720} & \textbf{17.8} & \textbf{0.0} &      6 &      44.1      & \textbf{720} &    36.4 & \textbf{0.0} & \textbf{5} & \textbf{45.3} &    9.5  &      1  \\
		& $(200,2000)$  &  480   & \textbf{480} & \textbf{18.5} & \textbf{0.0} &      \textbf{3} &      59.3      & \textbf{480} &    36.6 & \textbf{0.0} & \textbf{3} & \textbf{61.0} &    9.3  &      1  \\
		& $(300,3000)$  &  480   &  477 & \textbf{35.5} & \textbf{0.0} &    151 &      17.1      & \textbf{480} &    65.3 & \textbf{0.0} & \textbf{140} & \textbf{18.8} &   11.2  &      1  \\
		& $(400,4000)$  &  480   &  421 & \textbf{70.7} &       6.3 &    327 &      14.1      & \textbf{434} &   131.8 & \textbf{4.9} & \textbf{315} & \textbf{16.5} &   12.6  &      1  \\
		& $(500,5000)$  &  480   & \textbf{296} & \textbf{2181.8} &      19.4 & \textbf{50004} &       4.4      &  289 &  2328.0 & \textbf{17.4} &  50036 & \textbf{4.9} &    9.8  &      1  \\
		\bottomrule
	\end{tabular}
\end{table}

\subsection{Comparison with the state-of-the-art approaches}\label{sec:state-of-the-art}

In this subsection, we compare the performance of the proposed \BD algorithm with state-of-the-art approaches for solving \CCSCPs with and without the consideration of a  block structure of $\xi$.

We first consider the \CCSCP with an arbitrary finite discrete distribution of $\xi$, where the underlying formulation is problem \eqref{mip},
and compare the performance of the proposed \BD algorithm with the \cplex's default \BnC and automatic \BD algorithms (called \tbCPX and \tbAUTO, respectively).
The results are summarized in \cref{ijoc-cpxbd}. 
From \cref{ijoc-cpxbd}, we see that 
\rev{\tbCPX outperforms \tbBnBC on \CCSCP instances with $10^3$ scenarios.
	This is not surprising, as the \BD algorithm typically draws an advantage over the case for which the size of the Benders reformulation  can be significantly reduced compared to the original problem. 
	For \CCSCP instances with $10^3$ scenarios, the number of variables is only reduced by $10^3$, which does not allow the proposed \BD algorithm to draw a competitive advantage.}
\rev{However, for instances with at least $3\!\times\!10^3$ scenarios, as the number of variables in the Benders reformulation \eqref{mp1} is much smaller than that of the original problem  \eqref{mip},}
the performance of \tbBnBC is much better than that of \tbCPX.
Indeed, for large-scale instances with $10^6$ scenarios, \tbCPX was even unable to solve the \LP relaxation of problem \eqref{mip} \rev{within the time limit of 7200 seconds while \tbBnBC can solve $603$ instances among the $720$ instances to optimality}. 
Comparing \tbAUTO and \tbBnBC, we observe that 
\rev{for instances with at most $10^4$ scenarios, the performance of the two settings is comparable; \tbBnBC is slightly outperformed by \tbAUTO in terms of the \CPU time, but can solve more instances to optimality within the time limit of $7200$ seconds.}
For large-scale instances with $10^6$ scenarios, \tbBnBC significantly outperforms \tbAUTO; \tbBnBC can solve $25$ more instances with the \CPU time reduced by a factor of $4.6$.
	\rev{When looking at the computational results grouped by the numbers of rows and columns $(m,n)$, we observe that with the increasing of $(m,n)$, solving problem \eqref{mip} becomes more difficult for all three settings \tbCPX, \tbAUTO, and \tbBnBC; compared with the other two settings, our proposed \tbBnBC performs much better on the hard instances with $(m,n) =(400,4000), (500,5000)$.}

\begin{table}[t]
	\centering
	\small
	\renewcommand{\arraystretch}{1.1}
	\addtolength{\tabcolsep}{1pt}
	\caption{\small Performance comparison of \tbCPX, \tbAUTO, and \tbBnBC on problem \eqref{mip}. ``--'' means that \tbCPX was unable to find a feasible solution \rev{(for instances with $10^6$ scenarios)} within the time limit.}\label{ijoc-cpxbd}
	\begin{tabular}{llccccccccccccc}
		\toprule
		&    & \texttt{All}  &  \multicolumn{3}{c}{\tbCPX}   & \multicolumn{3}{c}{\tbAUTO}  & \multicolumn{3}{c}{\tbBnBC}    \\
		\cmidrule(r){4-6} \cmidrule(r){7-9}   \cmidrule(r){10-12}    
		&    &                &  \tbS & \tbT & \tbG          &\tbS & \tbT & \tbG           &\tbS & \tbT & \tbG  \\
		\hline
				$s$                   & $10^3$    &  720   &  \textbf{628} &   \textbf{61.3} &     \textbf{15.4}       &  585 & 85.2 &      21.4     & 609 &   103.0 & 18.6   \\
	& $3\!\times\!10^3$  &  720   &  593 &  177.8 &     21.8       &  585 & \textbf{91.8} &      18.9     & \textbf{606} &   101.4 & \textbf{16.8}   \\
	& $5\!\times\!10^3$  &  720   &  570 &  348.8 &     23.3       &  586 & \textbf{94.2} &      16.5     & \textbf{610} &   100.3 & \textbf{14.9}   \\
	& $10^4$    &  720   &  526 &  917.8 &     27.7       &  582 & \textbf{100.9} &      15.7     & \textbf{607} &   104.1 & \textbf{14.4}   \\
	& $10^6$    &  720   &   -- &     -- &       --      &  578 &   691.8 &      16.1      & \textbf{603} & \textbf{149.4} & \textbf{14.4}   \\
	\hline
	$(m,n)$               
	& $(50,500)$  &  300   &  239 &  160.0 &       --       & \textbf{300} & \textbf{19.9} & \textbf{0.0}     & \textbf{300} &    20.3 & \textbf{0.0}   \\
	& $(200,1000)$  &  900   &  720 &  122.4 &       --       & \textbf{900} & \textbf{21.2} & \textbf{0.0}     & \textbf{900} &    24.8 & \textbf{0.0}   \\
	& $(200,2000)$  &  600   &  480 &  127.7 &       --       & \textbf{600} & \textbf{26.2} & \textbf{0.0}     & \textbf{600} &    27.0 & \textbf{0.0}   \\
	& $(300,3000)$  &  600   &  472 &  496.0 &       --       & \textbf{600} &   107.6 & \textbf{0.0}     & \textbf{600} & \textbf{70.8} & \textbf{0.0}   \\
	& $(400,4000)$  &  600   &  378 & 1583.8 &       --       &  423 &   849.8 &       5.6     & \textbf{500} & \textbf{320.7} & \textbf{4.7}   \\
	& $(500,5000)$  &  600   &   28 & 7043.2 &       --       &   93 &  6257.4 &      20.5     & \textbf{135} & \textbf{4859.2} & \textbf{18.5}   \\
		\bottomrule
	\end{tabular}
	\vspace{-0.4cm}
\end{table}
\begin{table}[t]
	\centering
	\small
	\renewcommand{\arraystretch}{1.1}
	\addtolength{\tabcolsep}{0pt}
	\caption{\small Performance comparison of \tbPEPF of \cite{Saxena2010} and \tbBnBC on problem \eqref{minlp}.}\label{ijoc-pepfgbd}
	\begin{tabular}{llcccccccccc}
		\toprule
		& &  \texttt{All}  &  \multicolumn{5}{c}{\tbPEPF}   & \multicolumn{3}{c}{\tbBnBC}         \\
		\cmidrule(r){4-8} \cmidrule(r){9-11}   
		&       &       &   \tbS & \tbT & \tbG &\tbCON  &\tbET   &   \tbS & \tbT & \tbG \\
		\hline
		$s$                   & $10^3$    &  480   &  432 &    82.8 &      18.7 &   8488 &   50.2   & \textbf{467} & \textbf{40.5} & \textbf{12.7}    \\
		& $3\!\times\!10^3$  &  480   &  428 &    79.2 &      19.4 &   8080 &   55.0   & \textbf{466} & \textbf{40.1} & \textbf{12.2}    \\
		& $5\!\times\!10^3$  &  480   &  432 &    78.3 &      19.6 &   7923 &   56.9   & \textbf{465} & \textbf{40.5} & \textbf{11.9}    \\
		& $10^4$    &  480   &  432 &    85.0 &      20.0 &   7991 &   58.6   & \textbf{467} & \textbf{38.0} & \textbf{13.5}    \\
		& $10^6$    &  480   & \textbf{432} & \textbf{84.4} &      19.1 &   8099 &  293.0   &  431 &   107.3 & \textbf{14.8}    \\
		\hline
		$(m,n)$               & $(50,500)$  &  200   &  199 &    29.1 & \textbf{0.0} &   1548 &   23.9   & \textbf{200} & \textbf{16.6} & \textbf{0.0}    \\
		& $(200,1000)$  &  600   &  595 &    26.2 & \textbf{0.0} &   6516 &   66.3   & \textbf{600} & \textbf{18.0} & \textbf{0.0}    \\
		& $(200,2000)$  &  400   &  381 &    29.7 & \textbf{0.0} &   6516 &   68.5   & \textbf{400} & \textbf{17.6} & \textbf{0.0}    \\
		& $(300,3000)$  &  400   &  356 &    63.0 & \textbf{0.0} &   9807 &   86.9   & \textbf{400} & \textbf{26.7} & \textbf{0.0}    \\
		& $(400,4000)$  &  400   &  348 &   120.9 & \textbf{0.0} &  13111 &  106.4   & \textbf{400} & \textbf{41.6} & \textbf{0.0}    \\
		& $(500,5000)$  &  400   &  277 & \textbf{1750.3} &      19.3 &  16447 &  125.0   & \textbf{296} &  1943.7 & \textbf{13.2}    \\
		\hline
		$\epsilon$            & 0.05      & 1200   & \textbf{1190} & \textbf{24.1} & \textbf{10.1} &   2321 &   62.2   & 1163 &    49.4 &      12.3    \\
		& 0.1       & 1200   &  966 &   273.3 &      21.8 &  28354 &   95.6   & \textbf{1133} & \textbf{47.8} & \textbf{13.4}    \\
		\hline
		\texttt{Block size}       & 10        & 1200   & \textbf{1182} & \textbf{34.6} & \textbf{11.9} &   3649 &    2.2   & 1162 &    47.3 &      14.9    \\
		& 20        & 1200   &  974 &   192.1 &      21.2 &  18038 & 1897.7   & \textbf{1134} & \textbf{49.9} & \textbf{12.9}    \\
		\hline
		\texttt{Distribution}       & \texttt{Circular}  & 1200   & 1146 &    50.1 &      16.5 &   4907 &   55.3   & \textbf{1159} & \textbf{47.2} & \textbf{13.0}    \\
		& \texttt{Star}  & 1200   & 1010 &   133.5 &      20.7 &  13416 &  107.4   & \textbf{1137} & \textbf{49.9} & \textbf{13.3}    \\
		\bottomrule
	\end{tabular}
\end{table}

Next, we consider the \CCSCP with a block structure of $\xi$, where the underlying formulation is problem \eqref{minlp}, and  compare the performance of the proposed \BD algorithm with the state-of-the-art \BnC algorithm of \cite{Saxena2010}\rev{, denoted by \tbPEPF (in Online Appendix D, we also present computational results to compare the performance of the proposed \BD algorithm with the direct use of the state-of-the-art \MINLP solver Gurobi).}
The \BnC algorithm of \cite{Saxena2010} solves an \MIP formulation built on the so-called $(1-\epsilon)$-efficient and -inefficient points (which must be enumerated before starting the algorithm).
\cref{ijoc-pepfgbd} summarizes the computational results on instances with block sizes being $10$ and $20$.
We do not report the results on instances with a single block, as for the \BnC algorithm of \cite{Saxena2010},  the block size is too large to enumerate the $(1-\epsilon)$-efficient and -inefficient points within a reasonable period of time.
For setting \tbPEPF, we report the average number of (linear) constraints that are used to model the probabilistic constraint (\tbCON); and for completeness, we also report the average \CPU time spent in the enumeration phase (\tbET), which is not included in the time spent in solving the \MIP formulation. 
\rev{From \cref{ijoc-pepfgbd}, we observe that the performance of \tbPEPF and \tbBnBC is not sensitive to the number of scenarios $s$ when $s \leq 10^4$.
However, solving instances with $s=10^6$ using both \tbPEPF and \tbBnBC is more computationally expensive; for \tbPEPF, the additional computational effort comes from the enumeration procedure of $(1-\epsilon)$-efficient and -inefficient points.
As for the problem size, we observe that, as expected, with the increasing of $(m,n)$, solving problem \eqref{minlp} is more computationally expensive for both \tbPEPF and \tbBnBC; for the largest case $(m,n)=(500,5000)$, the proposed \tbBnBC  cannot solve all instances with $10^3$ scenarios (deduced from the facts that  \tbBnBC can solve \CCSCP instances with $(m,n) \leq (400, 4000)$ to optimality and  for instances with $10^3$ scenarios, it fails to solve $13$ of them to optimality).}
\rev{In addition,} 
when $\epsilon$ and the block size are small or $\xi^t$ is constructed by a circular distribution, the number of constraints in the \MIP formulation of \cite{Saxena2010} is small, and thus \tbPEPF performs better than \rev{or comparably to} the proposed \tbBnBC (provided that $(1-\epsilon)$-efficient and -inefficient points are at hand). 
However, for instances with a large $\epsilon$, a large block size, or the star distribution of $\xi^t$, 
due to the large number of constraints in the \MIP formulation, using \tbPEPF yields worse performance.
In contrast, as the proposed \BD algorithm does not need to solve a large \MIP formulation, it performs significantly  better than \tbPEPF on these instances, as illustrated in \cref{ijoc-pepfgbd}. 
Indeed, the performances of \tbBnBC on instances with different $\epsilon$, block size, or distribution of $\xi^t$ are comparable, showing that the performance of \tbBnBC is not sensitive to the parameters $\epsilon$, block size, or distribution of $\xi^t$.
Moreover, compared with \tbPEPF, the proposed \tbBnBC can avoid calling a potentially time-consuming enumeration procedure of $(1-\epsilon)$-efficient and -inefficient points.
These results show that for \CCSCPs with a block structure of $\xi$, the proposed \BD algorithm is much more robust than the state-of-the-art approach in \cite{Saxena2010} in terms of solving \eqref{biccscp} with different input parameters.

\begin{table}[t]
	\centering
	\small
	\renewcommand{\arraystretch}{1.2}
	\addtolength{\tabcolsep}{2.5pt}
	\caption{\small Performance comparison of problems \eqref{mip} and \eqref{minlp}.}\label{ijoc-bdgbd}
	\begin{tabular}{llcccccccccccc}
		\toprule
		&     &  \texttt{All}  &  \multicolumn{3}{c}{\texttt{Problem} \eqref{mip}}   & \multicolumn{3}{c}{\texttt{Problem} \eqref{minlp}} & \multicolumn{2}{c}{\texttt{Difference}} \\
		\cmidrule(r){4-6} \cmidrule(r){7-9} \cmidrule(r){10-11} 
		&      &       &   \tbS & \tbT & \tbG  &   \tbS & \tbT & \tbG &\tbNDO & \tbDO       \\
		\hline
		$s$             & $10^3$   &  720        &  609 &   103.0 &      16.9         & \textbf{660} & \textbf{57.8} & \textbf{16.0}      &   51 &    1.0    \\
		& $3\!\times\!10^3$ &  720        &  606 &   101.4 & \textbf{16.3}         & \textbf{658} & \textbf{57.6} &      17.1      &   32 &    0.9    \\
		& $5\!\times\!10^3$ &  720        &  610 &   100.3 &      16.3         & \textbf{656} & \textbf{57.9} & \textbf{16.1}      &   26 &    0.9    \\
		& $10^4$   &  720        &  607 &   104.1 &      16.7         & \textbf{660} & \textbf{54.7} & \textbf{16.6}      &   15 &    1.0    \\
		& $10^6$   &  720        &  603 &   149.4 &      18.0         & \textbf{622} & \textbf{135.6} & \textbf{15.5}      &    3 &    1.0    \\
		\hline
		\texttt{Block size} & 10       & 1200        & 1023 &   107.7 &      22.0         & \textbf{1162} & \textbf{47.3} & \textbf{13.8}      &   67 &    0.8    \\
		& 20       & 1200        & 1015 &   113.4 &      22.5         & \textbf{1134} & \textbf{49.9} & \textbf{13.1}      &   60 &    1.0    \\
		& $m$      & 1200        & \textbf{997} & \textbf{109.7} & \textbf{14.7}         &  960 &   131.8 &      17.8      &    0 &    0.0    \\
		\bottomrule
	\end{tabular}
\end{table}

\subsection{Comparison of solving problems \eqref{mip} and \eqref{minlp}}\label{sec:comparetwoproblem}

Finally, we compare the performance of problems \eqref{mip} and \eqref{minlp} when using them as \rev{sample approximation problems} to solve \CCSCPs.
We use the proposed \BD algorithms to solve the two problems.
\cref{ijoc-bdgbd} summarizes the computational results.
To see the difference between the two problems, we additionally report the number of instances with different optimal values (\tbNDO) and the gap between the optimal values of these instances 
(defined by $\tbDO = 100 \times \frac{|o_1 - o_2|}{\max\{o_1, o_2\}}$, where $o_1$ and $o_2$ represent the optimal values of the two problems).

As shown in \cref{ijoc-bdgbd}, for most cases, the optimal values of problems \eqref{mip} and \eqref{minlp} are identical, and even if the optimal values are different in some cases, their difference is usually very small. 
This shows that the gap between the two problems is indeed very small, especially for the cases with a huge scenario size.
As for the computational efficiency, we observe that, as expected, for instances with a single block, solving problem \eqref{minlp} is outperformed by solving problem \eqref{mip}.
In contrast,
for instances with multiple blocks (or block sizes of $10$ and $20$), problem \eqref{minlp} is much more solvable by the proposed \BD algorithm.
This is reasonable as by exploiting the block structure of the random variable $\xi$, up to $T$ cuts can be constructed in a single iteration, making the proposed \BD algorithm for problem \eqref{minlp} converge much faster.

\section{\rev{Conclusions and future work}}\label{section:conclusion}

In this paper, we have considered the \CCSCPs with an arbitrary finite discrete distribution of $\xi$ and with a block structure of $\xi$.
We developed efficient \BD algorithms for solving the \MIP and \MINLP formulations of the \CCSCPs where the scenario variables are projected out from the formulation and the Benders feasibility cuts are separated on the fly. 
Three key features of the proposed \BD algorithms, which make them particularly suitable for solving large-scale \CCSCPs, are:
(i) the numbers of variables in the underlying Benders reformulations are independent of the scenario sizes of $\xi$; 
(ii) the Benders feasibility cuts can be separated by an efficient polynomial-time algorithm;
and (iii) they are equipped with a class of strong valid inequalities---lifted cover inequalities, which renders them able to converge more quickly.
By extensive computational experiments, we demonstrated the effectiveness of our proposed \BD algorithms for solving large-scale \CCSCPs.
In particular, for \CCSCPs with an arbitrary finite discrete distribution of $\xi$, the proposed \BD algorithm is much more efficient than the \BnC and automatic \BD algorithms of \cplex~\rev{(when the number of scenarios is large)};
for \CCSCPs with a block structure of $\xi$, the proposed \BD algorithm is much more robust than the state-of-the-art approach in \cite{Saxena2010} in terms of solving problems with different input parameters.

\rev{As observed in \cref{sec:inequalities}, 
	the proposed approach for computing the general lifted cover inequalities in \cref{lciigeneral,lcispecial} is still computationally demanding, making it generally  not beneficial to the proposed \BD algorithms for solving \CCSCP{s}.
	As for future work, we shall develop acceleration strategies for efficiently solving the lifting problems.
	One interesting direction is to analyze the relations among the lifting problems and design warm-start algorithms based on the solutions of previously solved lifting problems, as to speed up the
	computation of the lifted cover inequalities.
	In addition, it would be interesting to extend the proposed \BD approach to other variants of the \CCSCP, 
	such as the robust version \citep{Lutter2017,Shen2023} and the conditional value-at-risk based risk-averse version \citep{Deng2018,Wu2023} of the \SCP.}

\vspace{0.7cm}
{\noindent\bf \rev{Acknowledgments}}\\[4pt]
\rev{We would like to thank the two anonymous reviewers for their insightful comments that significantly improved the quality of our work.}

\bibliography{shorttitles,pscp}
\bibliographystyle{apalike}

\end{document}